\documentclass[12pt]{article}
\newcommand{\version}{\today}

\usepackage{amsthm,amsfonts,amsmath, amscd}

\swapnumbers
                              
\theoremstyle{plain}
\newtheorem{thm}{THEOREM}[section]
\newtheorem{lm}[thm]{LEMMA}

\theoremstyle{definition}
\newtheorem{defi}{Definition}[section]

\newtheorem{example}{Example}[section]
\newtheorem{remark}{Remark}[section]

\if false

\theoremstyle{definition}
\newtheorem{defi}[thm]{DEFINITION}
\newtheorem{remark}[thm]{Remark}
\theoremstyle{remark}

\fi
\newcommand{\upchi}{\raise1pt\hbox{$\chi$}}
\newcommand{\R}{{\mathord{\mathbb R}}}
\newcommand{\C}{{\mathord{\mathbb C}}}

\newcommand{\hn}{{\mathord{\widehat{n}}}}

\newcommand{\F}{{\mathcal{F}}}

\renewcommand{\|}{{\Vert}}
\numberwithin{equation}{section}
\newcommand{\un}{{\rm 1\kern -2.5pt l}}

\def\subd{{\partial}}
\def\tnorm{|\!|\!|}

\begin{document}
\markboth{\scriptsize{CL \version}}{\scriptsize{CL February 20, 2015}}
\def\mn{{\bf M}_n}
\def\hn{{\bf H}_n}
\def\hnp{{\bf H}_n^+}
\def\hmnp{{\bf H}_{mn}^+}
\def\E{{\mathcal E}}
\def\F{{\mathcal F}}
\def\D{{\mathcal D}}
\def\S{{\mathcal S}}
\title{{\sc Duality and Stability for Functional Inequalities}}

\author{\vspace{5pt} Eric A. Carlen\\
\vspace{5pt}\small{ Department of Mathematics, Hill Center,}\\[-6pt]
\small{Rutgers University,
110 Frelinghuysen Road
Piscataway NJ 08854-8019 USA}\\
}

\date{\version}

\maketitle

\bigskip
\centerline{\it Dedicated to Dominique Bakry on the occasion of his 60th birthday}

\footnotetext [1]{Work partially supported by U.S.
National Science Foundation grant DMS 1501007.\hfill\break
\copyright\, 2016 by the author. This paper may be reproduced, in its
entirety, for non-commercial purposes.
}

\begin{abstract}  We develop a general framework for using duality to ``transfer''  stability results for 
a functional inequality to its dual inequality. As an application, we prove a stability bound for the Hardy-Littlewood-Sobolev inequality, which is related by duality, and the results proved here, to a stability inequality for the Sobolev inequality
proved by Bianchi and Egnell, and extended by Chen, Frank and Weth. We also discuss how the results proved here can be combined with the  proof of functional inequalities by means of flows to 
prove stability bounds with computable constants

\end{abstract}

\if fasle

\medskip
\leftline{\footnotesize{\qquad Mathematics subject classification numbers: 81V99, 82B10, 94A17}}
\leftline{\footnotesize{\qquad Key Words: uniform convexity, entropy }}

\fi



\maketitle

\section{Introduction}

Consider two functionals $\E$ and $\F$ on some normed real linear  space $X$ with values in $(-\infty,\infty]$ that are related by
a functional  inequality of the form 
\begin{equation}\label{in1}
\E(x) \leq \F(x) \quad{\rm  for\ all}\quad x\in X\ .
\end{equation}
We may regard any complex linear space as a real linear space by restriction of the linear structure. 
The functional inequality $\E \leq \F$ on $X$ is {\em sharp} if $\E(x) \leq \F(x)$ for all $x\in X$ and if for all $\lambda< 1$, there exist $x\in X$ such that $\E(x) > \lambda \F(x)$. 
The subset $X_0$ defined by 
\begin{equation}\label{in2}
X_0 = \{x\in X\ : \ \E(x)= \F(x) \quad{\rm and}\ \F(x) < \infty \}
\end{equation}
is called the {\em set of optimizers} of the inequality, When $X_0 \neq \emptyset$, the inequality is {\em optimal}.  An optimal functional inequality is necessarily sharp, but not {\em vice-versa}. 

Suppose that the inequality (\ref{in1}) is optimal, and suppose that $\{x_n\}$  is a sequence in $X$ such that
\begin{equation}\label{staa1}
\lim_{n\to \infty}\left(\F(x_n) - \E(x_n)\right) = 0\ .
\end{equation}
We are interested in conditions under which (\ref{staa1}) implies that  \
\begin{equation}\label{staa2}
\lim_{n\to\infty}d(x_n,X_0) = 0\ ,
\end{equation}
where $d$ is some metric on $X$ which may be different from the norm metric. We are also interested in how the rates of convergence in (\ref{staa1}) and (\ref{staa2}) may be related.  The following class of convex functions on $\R_+$
is pertinent to discussion of such rates of convergence. 

\begin{defi}[Rate function]\label{rate}  A {\em rate function} is a strictly convex $\Phi:[0,\infty)\to [0,\infty)$
with $\Phi(0) = 0$. 
\end{defi}

\begin{remark}\label{ratemin} The term {\em rate function} has a well established meaning in the context of 
large deviations problems.  Our usage here is not meant to suggest any connection. It is simply motivated
by the fact that theorems proved below will relate the rates of convergence in (\ref{staa1}) and (\ref{staa2}) in terms of such functions.
\end{remark}

\begin{remark}\label{ratemin2} 
Let $\Phi(t)$ be any rate function. Then, like all convex functions from $\R$ to $\R$, $\Phi$ is differentiable almost everywhere, and for all $t\in [0,\infty)$, 
$$\Phi(t) = \int_0^t \Phi'(s){\rm d}s\ .$$
Since $\Phi$ is strictly monotone increasing, $\Phi'> 0$ almost everywhere.  Hence, if 
$\Phi_1$ and $\Phi_2$ are rate functions,
$\min\{\Phi_1',\Phi_2'\} > 0$
almost everywhere, and then
$$\Phi_0 := \int_0^t  \min\{\Phi_1'(s),\Phi_2'(s)\}\ $$
is a rate function such that $\Phi_0(t) \leq \min\{\Phi_1(t),\Phi_2(t)\}$ for all $t$. Note that 
$\min\{\Phi_1(t),\Phi_2(t)\}$ need not be convex. 
\end{remark}

\begin{defi}[Stability bounds] Let $\E \leq \F$ be an optimal functional inequality on a normed linear space $X$, and let
$X_0$ be the set of optimizers.   Let $d$ be a metric on $X$, not necessarily the metric induced by the norm, and let
$\Phi$ be a rate function. Then the functional inequality $\E \leq \F$ is {\em $(d,\Phi)$-stable} in case
\begin{equation}\label{dphi1}
\F(x) - \E(x) \geq \Phi(d(x,X_0)) \quad {\rm for\  all}\quad  x\in X\ .
\end{equation}
In the special case that $d(x,y) = \|x-y\|_X$ and $\Phi(t) = \kappa t^2$ for some $\kappa>0$ and all $t\in [0,\infty)$, so that
\begin{equation}\label{dphi2}
\F(x) - \E(x) \geq \kappa \inf_{x\in X_0}\|x-z\|_X^2  \quad {\rm for\  all}\quad  x\in X\ ,
\end{equation}
we say the the functional inequality $\E \leq \F$ is {\em $\kappa$-stable}. 
\end{defi} 

Evidently, when the optimal functional inequality is $(\Phi,d)$-stable, the rate of convergence in (\ref{staa1}) governs the rate of convergence in (\ref{staa2}). 

$L^p$ norms on a measure space $(\Omega,\mathcal{B},\mu)$ occur frequently in the examples that illustrate our results as well as in their applications. Throughout the exposition, for all $1\leq p < \infty$, $\|f\|_p$ denotes $\left(\int_\Omega |f|^p {\rm d}\mu\right)^{1/p}$, and $\|f\|_\infty$ denotes the essential supremum of $|f|$.  

Before further developing the abstract theory, we turn to a concrete example, namely the Sobolev inequality.  Let $n\geq 3$ and let $0 < \alpha < n/2$.  There is a unique domain
in $L^2(\R^n)$ containing the smooth, compactly supported functions on which $(-\Delta)^{\alpha/2}$ is self-adjoint, where
$\Delta$ is the Laplacian on $\R^n$.  See Lieb and Loss \cite{LL}  for a concrete description of this domain.  (For instance,
$(-\Delta)^{-\alpha/2}$ may be written in term of the Fourrier transform, and also as the composition of a Riesz potential with an integer power of the Laplacian.)

The Sobolev inequality in its optimal form says that  
 for all $f$ in the domain of $(-\Delta)^{\alpha/2}$, 
\begin{equation}\label{sob1}
\left(\int_{\R^n} |f(\eta)|^{2n/(n-2\alpha)}{\rm d}\eta \right)^{(n-2\alpha)/n} \leq \S_{n,\alpha}
\int_{\R^n} |(-\Delta)^{\alpha/2} f(\eta)|^2{\rm d}^n\eta \ .
\end{equation}
where
\begin{equation}\label{sob1a}\S_{n,\alpha}  = \frac{ \| h\|_{2n/(n-2\alpha)}^2}{ \|(-\Delta)^{\alpha/2} h\|_2^2} \quad{\rm and}\quad 
h(\eta) = (1+|\eta|^2)^{-(n-2\alpha)/2}\ .
\end{equation}
Evidently there is equality in (\ref{sob1}) when $f = h$. More generally, taking into account the homogeneity, and the translation and dilation invariance of the inequality, there is equality 
when $f$  is a multiple of a function of the form $h(a(\eta - \eta_0))$ for some $a>0$ and some $\eta_0\in \R^n$. In fact, these are the {\em only} cases of equality. 

In the special case $\alpha=1$, for which we have $\|(-\Delta)^{1/2}f\|_2^2 = \|\nabla f\|_2^2$, this result was proved by Aubin \cite{A} and Talenti \cite{T}. The general case is a direct consequence of Lieb's theorem \cite{L} on the optimal 
Hardy-Littlewood-Sobolev (HLS)  inequality, to which (\ref{sob1}) is equivalent by duality as we recall below. 
The integrals  in $ \| h\|_{2n/(n-2\alpha)}^2$ and $ \|(-\Delta)^{\alpha/2} h\|_2^2$ are easily evaluated in terms of $\Gamma$ functions, but the explicit value plays no role in what follows.

To relate this example to the general framework introduced above, 
let $X$ denote the completion of smooth, compactly supported functions $f$
on $\R^n$ in the norm   $\tnorm \cdot \tnorm_X$   defined by 
\begin{equation}\label{tnorm1}
\tnorm f\tnorm_X := \|(-\Delta)^{\alpha/2} f\|_2 \ .
\end{equation}

Define functionals $\E$ and $\F$ on $X$ by
\begin{equation}\label{tnorm1B}
\E(f) :=  \|f\|_{2n/(n-2\alpha)} \qquad{\rm and}\quad 
 \F(f) := \S_{n,\alpha}\tnorm f\tnorm_X^2 \ .
\end{equation}
 Then we can rewrite
(\ref{sob1}) in the form $\E(f) \leq \F(f)$.   By the characterization of the optimizers quoted above, this
 inequality is optimal, and the set $X_0$ of optimizers is given by
\begin{equation}\label{sob1c}
X_0 = \{ z h(a(\eta - \eta_0)) \ :\ z\in \C\ , a> 0\ , \eta_0\in \R^n\  \}\ .
\end{equation}

A line of work starting with  Bianchi and Egnell  \cite{BE} for the case $\alpha =1$, 
has addressed the stability of the Sobolev inequality. 
The recent paper of Chen, Frank and Weth \cite{CFW} extends the result of Bianchi and 
Egnall to the full range of $\alpha$.
Theorem 1 of \cite{CFW} states that 
 that for all $\alpha \in (0,n/2)$,  there exists some $\kappa_{BE}>0$ depending only on $n$ and $\alpha$ such that 
 that for all $f\in X$,
\begin{equation}\label{BE}
\F(f) - \E(f) \geq \kappa_{BE} \inf_{z\in X_0}\tnorm f-z\tnorm_X^2\ .
\end{equation}
In our terminology, this says that  the  Sobolev inequality $\kappa_{BE}$-stable.  

The papers \cite{BE} and \cite{CFW} are separated by 22 years, and in-between a number of authors treated various integer values of $\alpha$; see \cite{CFW} for the history. The breakthrough to general $\alpha$ in \cite{CFW} turns on an insightful use of the stereographic projection to lift the problem to the sphere, which was also crucial to Lieb's determination of the sharp form of the HLS inequality, though the use made of this in \cite{CFW} is somewhat different; it facilitates an eigenvalue computation that has no analog in Lieb's work. 

All of the work discussed in the previous paragraph follows the general strategy initiated by Bianchi and Egnell of proving a ``local'' stability bound for functions that are sufficiently close to the set of optimizers, and then 
using  a compactness argument to show by contradiction that there are no ``almost optimizers'' that are far from the set of optimizers, and thus to  obtain the globally valid result (\ref{BE}).  It is on account of this
 compactness argument that these proofs do not give give any means of explicitly estimating $\kappa_{BE}$ from below.   

Our main concern in this paper is with stability theorems for functional inequalities, such as the ones just discussed, and how they may be
 proved using {\em duality arguments} and {\em flow arguments}. 

We begin with duality. As noted above, it has been known for some time
that the optimal form of the Sobolev inequality is {\em dual} to Lieb's optimal 
Hardy-Littlewood-Sobolev (HLS) inequality; see e.g. \cite{CL}, as we now recall.  
 
 Define $Y = L^{2n/(n+2\alpha)}(\R^n)$.  For $g\in Y$, define $\E^*(g) = \|g\|_{2n/(n+2\alpha)}^2$. 
  It is well known, and easy to check, that
 if we define, for all $f\in L^{2n/(n-2\alpha)}(\R^n)$ and all $g\in L^{2n/(n+2\alpha)}(\R^n)$, a real bilinear form by
 \begin{equation}\label{dufoA}
 \langle f,g\rangle = 2\Re\left( \int_{R^n} f^*g {\rm d}x\right)\ ,
 \end{equation}
 then, using the notation of the previous example,
 \begin{equation}\label{dufoB}
 \E^*(g) = \sup_{f\in  L^{2n/(n-2\alpha)}(\R^n)} \left\{  \langle f,g\rangle - \|f\|_{2n/(n-2)}^2 \right\} = 
 \sup_{f\in X} \left\{  \langle f,g\rangle -\E(f) \right\} \ .
 \end{equation}
 The reason for the factor of 2 in the bilinear form on (\ref{dufoA}) is that it avoids factors of $1/2$ in the 
 definitions of $\E$ and $\E^*$ that would otherwise be required and are inconvenient in our context. The calculations
 justifying this are explained below when we recall facts about the Legendre transform; (\ref{dufoB}) says that 
 {\em $\E^*$ is the Legendre transform of $\E$}. 
 We point out that the second equality in (\ref{dufoB}) is valid since the Sobolev space $X$, defined in 
 the previous example, is dense in $L^{2n/(n-2\alpha)}(\R^n)$. 

Now define a functional $\F^*$ on $Y$ by
\begin{equation}\label{dufoC}
 \F^*(g) = \sup_{f\in  X} \left\{  \langle f,g\rangle - \F(f) \right\} = 
 \sup_{f\in X} \left\{  \langle f,g\rangle - \S_{n,\alpha}\|(-\Delta)^{\alpha/2}f\|_2^2 \right\} \ .
 \end{equation}
A simple completion of the square now shows that
$\F^*(g) = \S_{n,\alpha}^{-1} \|(-\Delta)^{-\alpha/2}f\|_2^2$.
The right hand side may be written in terms of the Green's function for  $(-\Delta)^\alpha$; that is
\begin{equation}\label{green}
\|(-\Delta)^{-\alpha/2}f\|_2^2 = \mathcal{C}_{n,\alpha}\int_{\R^n}\int_{\R^n} g(x) \frac{1}{|x-y|^{n-2\alpha}} g(y){\rm d}x {\rm d}y
\end{equation}
where an explicit expression for the constant $\mathcal{C}_{n,\alpha}$ may be found in \cite{LL}. This formula has the advantage that the right hand side is well-defined for any non-negative measurable function $g$. 

By the Sobolev inequality, for all $f\in X$ and all $g\in Y$, 
$$\langle f,g\rangle - \F(f) \leq   \langle f,g\rangle - \E(f)\ .$$
Taking the supremum in $f$ over $X$ on both sides, we obtain
\begin{equation}\label{dufoA1}
\F^*(g) \leq \E^*(g) \qquad {\rm for\ all}\quad g\in Y\ .
\end{equation}
This is the well known {\em order reversal} property of the Legendre transform. Applying the Legendre transform to a functional inequality $\E \leq \F$ yields a new inequality for the Legendre transforms, namely $\F^* \leq \E^*$. 

The inequality (\ref{dufoA1}) is Lieb's optimal LHS inequality. More explicitly, it states that
\begin{equation}\label{hlsA}
\|(-\Delta)^{-\alpha/2} g\|_2^2 \leq \S_{n,\alpha}\| g\|_{2n/(n+2\alpha)}^2\ .
\end{equation}
The fact that the same constant $\S_{n,\alpha}$ shows up in (\ref{hlsA}) and in (\ref{sob1}) is no accident; it is essentially a consequence of the involutive nature of the Legendre transform: Taking the dual of $\E \leq \F$, we obtain $\F^* \leq \E^*$. If for some $\lambda < 1$, we could improve this to $\F^* \leq \lambda \E^*$, then taking the dual again we would obtain  $\E^{**}\leq \lambda \F^{**}$. Under mild regularity conditions specified in the Fenchel-Moreau Theorem
discussed below, the Legendre transform is involutive  so that $\E^{**} = \E$
and $\F^{**} = \F$. Thus we would have $\E \leq \lambda \F$, but this is impossible when $\E \leq  \F$ is optimal. In other words, when when the Legendre transform acts as an involution, it takes sharp inequalities to sharp inequalities. In particular, (\ref{sob1}) and (\ref{hlsA}) are equivalent inequalities,   which for general $\alpha$ were first found by Lieb \cite{L} in the form (\ref{hlsA}).

Moreover, the duality relating (\ref{sob1}) and (\ref{hlsA}), as well as other dual pairs of optimal inequalities, goes much further. As we show here, when $\E \leq \F$ is not only optimal, but stable, then under suitable quantitative convexity assumptions on $\E$, one obtains both optimality and a stability result for the dual inequality $\F^* \leq \E^*$.  As a particular consequence, we shall obtain the following stability theorem for the HLS inequality:

\begin{thm}\label{HLSstth} For all $\alpha \in (0,n/2)$
There is a constant $\kappa_{BE}^*> 0$ depending only on $n$ and $\alpha$ such that for all $g\in L_{2n/(n-2\alpha)}(\R^n)$,
 \begin{equation}\label{HLSB3}
  \S_{n,\alpha}\| g\|_{2n/(n+2\alpha)}^2 - \|(-\Delta)^{-\alpha/2} g\|_2^2 \geq   
  \kappa_{BE}^*\inf_{y\in  Y_0}\| g - y\|_{2n/(n+2\alpha)}^2\ ,
  \end{equation}
  and the constant $\kappa_{BE}^*> 0$ is explicitly computable in terms of the constant $\kappa_{BE}$ in (\ref{BE}). 
\end{thm}

This result is new. If one attempts to prove it directly using the  method of Bianchi and Egnel, the non-locality of $(-\Delta)^{-\alpha/2}$ complicates matters. Perhaps these complications  could be managed, but as shown here, there is no need for this.

Notice that just as in (\ref{BE}), the norm in the remainder term is the {\em stronger} of the two norms involved in the functional inequality.

Duality methods for functional inequalities are useful because one or the other of the dual forms might be easier to prove directly. Indeed, except for the $\alpha =1$ cases, Lieb's optimal form of the HLS inequality (\ref{hlsA}) predates the optimal Sobolev inequality (\ref{sob1}). Lieb's proof of the HLS inequality makes important use of rearrangement inequalities. In particular,
by the Riesz Rearrangement Inequality and (\ref{green}), $\|(-\Delta)^{-\alpha/2}f\|_2^2$ does not decrease when $f$ is replaced by its spherical decreasing rearrangement. This is also true of $\|(-\Delta)^{-\alpha/2}f\|_2^2$ for $\alpha \leq 1$, but only for 
$\alpha \leq 1$. A pair of functional inequalities may be equivalent via the Legendre transform, but a wider range of tools and techniques may be applicable to one or the other of them. 

As we shall see, this is also true in proving stability inequalities. One tool that is useful for this purpose is {\em monotone flows}.
Let $\E \leq \F$ be an optimal functional inequality on $X$. Let $\{\Phi_t\}$ be a semigroup of transformations on $X$. Suppose  that for all $x$ in $X$, $\lim_{t\to\infty}\Phi_t(x)$ exists and belongs to $X_0$. Suppose also that for all $t>s \geq 0$,
$$\F(\Phi_t(x)) - \E(\Phi_t(x)) \leq \F(\Phi_s(x)) - \E(\Phi_s(x)) \ .$$
Then the flow associated to the semigroup $\{\Phi_t\}$ is {\em monotone for the functional inequality $\E \leq \F$}. 
Under certain additional conditions, monotone flows may be used to prove stability inequalities. Examples of this 
may be found in \cite{CF}. Here we shall combine such flow methods with duality to prove variants of the stability 
bounds for the Sobolev and HLS inequalities
with computable constants. 

The transfer of a stability bound for $\E \leq \F$ to its dual $\F^* \leq \E^*$ is not so 
simple as the almost automatic order reversal property of the Legendre transform that has been described above. 
Indeed, one way to write a stability bound for $\E \leq \F$ is 
$\E + \mathcal{R} \leq F$ where $\mathcal{R}$ is a ``remainder term''. For example, in (\ref{BE}) we would have 
$\mathcal{R}(f) = \kappa_{BE} \inf_{z\in X_0}\tnorm f-z\tnorm_X^2$. The Legendre transform of a sum is the 
{\em infimal convolution} of the Legendre transforms of the summands; see, e.g., \cite{rock}. That is, going back to the general 
setting of functional on a dual pair $(X,Y,\langle\cdot, \cdot\rangle)$, 
$$(\E + \mathcal{R})^*(y)  = \inf_{z\in Y} \{ \E^*(y-z) + \mathcal{R}^*(z)\}\ .$$
Doing this for (\ref{BE}) does not yield anything at all like (\ref{HLSB3}), or even anything that is likely to be at all useful. 
The passage from (\ref{BE}) to (\ref{HLSB3}) is more complicated, and involves an interplay between the stability bound (\ref{BE}) an {\em quantitative convexity estimates} for the functional $\E$, as will be explained below. This interplay between stability bounds and quantitative convexity bounds was crucial in the paper \cite{CFL} on stability of optimal potentials for Schr\"odinger operators. In \cite{CFL} an ``adding  and subtracting'' argument is used to combine a Bianchi-Egnell like stability bound for a family of Gagliardo-Nirenberg inequalities with a stability bound for H\"older's inequality to  prove a stability bound for an inequality relating the fundamental eigenvalue $\lambda(V)$ of a Schr\"odinger operator $-\Delta + V(x)$ with
the $L^q$ norm of $V$ for appropriate $q$. The eigenvalue inequality is the Legendre transform dual of the 
Gagliardo-Nirenberg inequality. The present work was partly motivated by the desire to understand this transfer of stability within a general framework of duality.

\medskip

\noindent{\bf \large Acknowledgements}

\medskip

I would like to thank Dominque Bakry for the invitation to present a three lecture mini-course
on functional inequalities and evolution equations in Toulouse in April 2014. The main ideas (and some other examples)
were presented during this course. Dominque Bakry, Szymon Peszat and 
Boguslaw Zegarlinski then arranged for me to give a second course on the topic 
at the CNRS-PAN Mathematics Summer Institute, June 28-July 4, 2015 in Krakow.
The applications to  quantitative stability for the Sobolev inequality were carried out with these lectures in mind. This work was partially done while I was visiting the I.M.A. in Minneapolis in the spring of 2015. I thank the I.M.A. for a providing a stimulating and congenial environment for this work. Finally, I would like to thank Rupert Frank and Elliott Lieb for many discussions about stability, and Rupert Frank and two anonymous referees for a careful reading of the first version of this work and for making  valuable comments.

\section{Duality } 

\begin{defi}\label{dualp} Let $X$ and $Y$ be two normed spaces with norms $\|\cdot \|_X$ and $\|\cdot \|_Y$
respectively. 
Let $\langle \cdot,\cdot \rangle$ be a separately continuous bilinear form on $X\times Y$ such that for all
non-zero $x\in X$, there exists $y\in Y$ with  $\langle x,y\rangle>0$ and likewise, for all non-zero $y\in Y$, there exists $x\in X$ 
with  $\langle x,y\rangle>0$.  Then the triple $(X,Y, \langle \cdot,\cdot \rangle)$ is a {\em dual pair
 of normed linear spaces}.
\end{defi}

When $(X,Y, \langle \cdot,\cdot \rangle)$ is a dual pair, the map $x\mapsto \langle x,\cdot\rangle$ is injective into
$Y^*$, the topological dual to $Y$. Likewise, $y\mapsto \langle \cdot,y\rangle$ is injective into
$X^*$, the topological dual to $X$.
Throughout this section $X$ and $Y$ together with the bilinear form $\langle\cdot,\cdot\rangle$ is a fixed dual pair of normed linear  spaces.  Most of the material well-known, and can be found in Rockafellar \cite{rock}, for example. However,  we include some proofs since they are very short and relevant to what follows. 

A convex function $\E$ from $X$ to $(-\infty,\infty]$  is {\em proper} in case it is not identically $+\infty$. 
The {\em Legendre transform} of a convex functional $\E: X \to (-\infty,\infty]$ is the function $\E^*$ on $Y$ defined by
\begin{equation}
\E^*(y) = \sup_{x\in X} \{\  \langle x,y\rangle - \E(x)\ \}\ .
\end{equation}
The function $\E^*$, being a supremum of a family of continuous linear functions is lower semicontinuous and convex. 

\begin{defi}[Closed convex function] \label{ccf}
Let $X$ be a normed linear space. A function $\F$ on $X$ with values in $(-\infty,\infty]$ 
is a {\em closed convex function} in case it is lower semicontinuous.
\end{defi}

The terminology is motivated by the fact that a convex function  $\F$ is lower semicontinuous
 if and only if its {\em epigraph} is closed; see \cite{rock}.

The {\em Fenchel-Moreau Theorem} states that if $\E$ is a proper closed  convex function on $X$, 
and $(X,Y, \langle \cdot,\cdot, \rangle)$ is the canonical  dual pair associated to $X$,
then $\E^{**} = \E$
where $\E^{**}$ denotes the Legendre transform of $\E^*$. 

It is a direct consequence of the definition that  $\langle x,y\rangle \leq \E(x) + \E^*(y)$ for all $x\in X$ and  all $y\in Y$.
This is {\em Young's inequality}. The cases of equality in this inequality are crucially important for what follows. This brings us to
the notion of the  subgradient of a convex function: 

\begin{defi}[subgradient]  Let $(X,Y, \langle \cdot,\cdot \rangle)$
be a dual pair, an let $\E$ be a convex function on $X$. The {\em subgradient of $\E$ at $x_0\in X$} is the subset 
$\subd\E(x_0)$ of $Y$ consisting of all $y$ such that
\begin{equation}\label{subg}
\E(x_1) \geq \E(x_0)  + \langle x_1-x,y\rangle \qquad{\rm for\ all}\quad x_1\in X\ .
\end{equation}
For $W\subset X$, er define $\subd \E(W) = \cup_{x\in W}\subd \E(x)$. 
\end{defi}
That is, $y \in \subd\E(x_0)$ if and only if the  hyperplane $(x,\E(x_0) + \langle x-x_0,y\rangle )$ lies below the graph of $x\mapsto \E(x)$, and intersects it at $(x_0,\E(x_0))$, so that the hyperplane $(x,\E(x_0) + \langle x-x_0,y\rangle )$ is a supporting plane to the graph of  $x\mapsto \E(x)$ at $x_0$. At this level of generality, $\subd\E(x)$
may be empty.
However, it is easy to see that if $\E$ is differentiable at $x$, with derivative $\nabla \E(x)\in Y$,  then $\subd\E(x) = \{ \nabla \E(x)\}$. 

The domain of a proper, closed convex functional $\F$ on $X$ is the set $\{ x\in X \:\ \F(x) < \infty\ \}$. As is well known, for a proper closed convex function $\F$, $\subd \F(x)$ is never empty at any $x$ in the domain of $\F$; see \cite{rock}.  In particular, when we have an optimal  functional inequality $\E \leq \F$ relating two  proper, closed convex functionals, the set of optimizers is, by
the definition (\ref{in1}), included in the domains of both $\E$ and $\F$.

\begin{example}[Squared $L^p$ norms]\label{lpex}  Let $(\Omega, \mathcal{B},\mu)$ be a measure space. Let $1 < p < \infty$, and let 
$p' = p/(p-1)$. Let $X = L^p(\Omega, \mathcal{B},\mu)$ regarded as a real Banach space (so that for nonzero $f \in 
L^p(\Omega, \mathcal{B},\mu)$, $f$ and $if$ are linearly independent). Let $Y = L^{p'}(\Omega, \mathcal{B},\mu)$
regarded as a real Banach space, and define $\langle \cdot, \cdot \rangle$ on $X\times Y$ by
$$\langle f,g\rangle = 2\Re \left(\int_\Omega f^* g{\rm d}\mu \right)\ .$$
Evidently with this bilinear form, $X$ and $Y$ are a dual pair.  Consider the functional $\E$ on $X$ defined by 
$\E(f) = \|f\|_p^2$.
Since $t\mapsto t^2$ is convex and increasing on $[0,\infty)$, it is a consequence 
of the Minkowski inequality that $\E$ is convex. 

By H\"older's inequality, for all non-zero $f\in X$ and $g\in Y$,
$$\langle f,g\rangle - \E(f)   = 2\Re \left(\int_\Omega f^* g{\rm d}\mu \right) - \|f\|_p^2 \leq 2\|g\|_{p'}\|f\|_p - \|f\|_p^2 = 
\|g\|_{p'}^2 - (\|g\|_{p'} - \|f\|_{p})^2\ .$$
By the well-known cases of equality in H\"older's inequality, there is equality if and only if 
$g$ is a multiple of $|f|^{p-1}{\rm sgn}f$. We conclude that
$\langle f,g\rangle - \E(f)   \geq \|g\|_{p'}^2$
with equality if and only if 
\begin{equation}\label{lpdu}
g = \|f\|_p^{2-p}|f|^{p-1}{\rm sgn}(f)\ .
\end{equation}
This proves that $\E^*(g) = \|g\|_{p'}^2$, and that $\subd\E(f) = \{\|f\|_p^{2-p}|f|^{p-1}{\rm sgn}(f)\}$ for non-zero $f$. (It is evident
that $\subd\E(0) = \{0\}$. In this example  $\E$ (together with $\E^*$ by symmetry) is differentiable, 
and $\subd\E(f) = \{\nabla \E(f)\}$ where $\nabla \E(f) = \|f\|_p^{2-p}|f|^{p-1}{\rm sgn}(f)$.

Finally, note that $\nabla \E$ is homogeneous of degree $1$, and is a norm preserving map from $X$ to $Y$. That is,
\begin{equation}\label{isom}
\|\nabla \E(f)\|_{p'} = \|f\|_p\ .
\end{equation}
\end{example}

\begin{lm}\label{sgcont} Let $\E$ and $\F$ be proper lower 
semicontinuous convex functions on $X$ and suppose that the inequality $\E \leq \F$ is valid and optimal. Let $X_0$ be the set of optimizers for it. Then for all $x_0\in X_0$,
\begin{equation}\label{in7}
\subd\E(x_0) \subset  \subd\F(x_0) \ .
\end{equation}
\end{lm}

\begin{proof} If $y_0\in \subd\E(x_0)$, then for all $x_1\in X$,  $\E(x_1) \geq \E(x_0) + \langle x_1- x_0,y_0\rangle$.
Now using $\F(x_1) \geq \E(x_1)$ and $\F(x_0) = \E(x_0)$, we have
\begin{eqnarray}
\F(x_1) \geq \E(x_1) &\geq& \E(x_0) + \langle x_1- x_0,y_0\rangle\nonumber\\
&=& \F(x_0) + \langle x_1- x_0,y_0\rangle\ ,\nonumber
\end{eqnarray}
and this proves $y_0\in  \subd\F(x_0)$.
\end{proof}

\begin{example}\label{sqex} Let $X= Y =\R$, and let $\langle a,x\rangle = xy$. Let $\E(x) = |x|$ and $\F(x) = 2|x|$.  Then $\E \leq \F$ and
$X_0 = \{0\}$, and we have $\subd \E(0) =[-1,1]$ and $\subd\F(0) = [-2,2]$, showing that the containment in Lemma~\ref{sgcont} can be strict. 
\end{example}

\begin{lm}[Young's inequality]\label{youngin}  Let $(X,Y,\langle \cdot,\cdot\rangle)$ be a dual pair.
Let $\E$ be a closed  convex function on $X$. Then for all $x\in X$ and all $y\in Y$,
\begin{equation}\label{young}
\langle x,y\rangle \leq \E(x) + \E^*(y)\ .
\end{equation}
Moreover:

\smallskip
\noindent{\it (1)}  Whenever there  is equality in (\ref{young}),  $y\in \subd\E(x)$ and  $x\in \subd\E^*(y)$.

\smallskip
\noindent{\it (2)}  Whenever  $y\in \subd\E(x)$, there  is equality in (\ref{young}), and consequently
\begin{equation}\label{youngGF}
y\in \subd\E(x) \quad \Rightarrow \quad  x\in \subd\E^*(y)\ ,
\end{equation}
and in case $\E^{**} = \E$, the reverse implication is also valid. 
\end{lm}

\begin{proof}
We need only explain the cases of equality.  If there is equality in (\ref{young}), then for all $y_1\in Y$,
$$\langle x,y_1 \rangle  - \langle x,y\rangle \leq (\E(x) + \E^*(y_1)) - (\E(x) + \E^*(y))  = \E^*(y_1) - \E^*(y)\ ,$$
which is the same as
$\E^*(y_1) \geq \E^*(y) + \langle x,y_1- y\rangle$, which means that $y\in \subd\E(x)$. 
An analogous argument shows that $x \in \partial \E^*(y)$ when there is equality in (\ref{young}).

Conversely, suppose
that $y\in \subd\E(x)$. Then for all $x_1\in X$,
$\E(x_1) \geq \E(x) + \langle x_1- x,y\rangle$, and rearranging terms we obtain
$$\langle x,y\rangle \geq \E(x) + \langle x_1,y\rangle - \E(x_1)\ .$$
Taking the supremum over $x_1$, we obtain $\langle x,y\rangle \geq \E(x) + \E^*(y)$, and equality must hold in (\ref{young}).  An analogous argument shows that  argument  when $x \in \partial \E^*(y)$,
then $\langle x,y\rangle \leq \E^{**}(x) + \E^*(y)$,
which is the same a (\ref{young}) when $\E^{**} = \E$.
\end{proof}

\subsection{Subgradients and optimizers}

It is well-known that the Legendre transform reverses order, so that if $\E(x) \leq \F(x)$ for all $x\in X$, then $\F^*(y) \leq \E^*(y)$ for all $y\in Y$. As mentioned in the introduction, one way to see this is to note that  $\E(x) \leq \F(x)$  implies that
$\langle x,y\rangle - \F(x) \leq  \langle x,y\rangle - \E(x)$.
Taking the supremum over $x$,  we obtain $\F^*(y) \leq \E^*(y)$.  

By the Fenchel-Moreau Theorem, if $\E$ and $\F$ be proper, closed convex  functions on $X$, the Legendre transform is involutive, and in this case Legendre transforming twice yields:
\begin{equation}\label{rev27}  
  \E(x)   \leq \F (x)\quad {\rm for\ all}\ x\in X \quad \iff \quad   \F^*(y) \leq  \E ^*(y)
  \quad {\rm for\ all}\ y\in Y\ ,
 \end{equation}
which expresses the equivalence of the primal inequality $\E \leq \F$ and its dual inequality $\F^*\leq \E^*$. 

However, something more is true: Under the condition that  $\E$ and $\F$ are proper, closed convex  functions on $X$,
there is a general relation
between the optimizers (if any) of $\E \leq \F$ and the optimizers (if any) of $\F^* \leq \E^*$.

\begin{thm}[Duality for functional inequalities]\label{reverse} Let $(X,Y,\langle \cdot,\cdot\rangle)$ be a dual pair. 
Let $\E$ and $\F$ be proper, closed convex  functions on $X$ so that
$\E^{**} = \E$ and $\F^{**} = \F$.
Then for all 
$y\in \subd\F(x)$, 
\begin{equation}\label{rev1}  
\F(x) - \E(x) \leq  \E^*(y) - \F^*(y) \ ,
\end{equation}
and for all $x\in \subd\E^*(y)$,
\begin{equation}\label{rev17}  
 \F(x) - \E(x) \geq  \E^*(y) - \F^*(y)\ .
\end{equation}
Suppose further that $\E \leq \F$, or equivalently, $\F^* \leq \E^*$. 
If $X_0$ is the set of optimizers of the inequality  $\E  \leq \F$ and  $Y_0$ is the 
set of optimizers of the inequality  $\F^*  \leq \E^*$, then
\begin{equation}\label{in8b}
Y_0 = \subd\E(X_0) \qquad{\rm and}\qquad  X_0 = \subd\F^*(Y_0) \ .
\end{equation}
\end{thm}

\begin{proof} Let $y\in \subd\F(x)$. By Young's inequality, and then the cases of equality in it,
$$\F(x) + \F^*(y) = \langle x,y\rangle  \leq  \E(x) + \E^*(y)\ . $$
Rearranging terms we obtain (\ref{rev1}).   Likewise, let $x\in \subd\E^*(y)$. Then
$$\F(x) + \F^*(y) \geq  \langle x,y\rangle  =  \E(x) + \E^*(y)\ . $$
and rearranging terms, yields (\ref{rev17}).

Now suppose that $X_0 \neq \emptyset$, and $x_0\in X_0$. By the remark preceding Example~\ref{lpex}, $\subd\E(x_0) \neq \emptyset$, and by Lemma~\ref{youngin}, for any $y_0 \in 
\subd\E(x_0)$, $x_0\in \subd\E^*(y_0)$, so that (\ref{rev17}) implies that  $0 \geq \E^*(y_0) - \F^*(y_0) $, so that $y_0\in Y_0$. 
This shows that $\subd\E(X_0)\subset Y_0$.  

Likewise, suppose that $Y_0 \neq \emptyset$, and $y_0 \in Y_0$. Then, as above, $\subd\F^*(y_0)\neq \emptyset$, and
 for any $x_0 \in 
\subd\F^*(y_0)$, $y_0\in \subd\F(x_0)$, so that (\ref{rev1}) implies that  $0 \geq \F(x_0) - \E(x_0) $, so that $x_0\in X_0$. 
This shows that $\subd\F^*(Y_0)\subset X_0$.

We next show that $Y_0 \subset \subd\E(X_0)$. Suppose that $y_0\in Y_0$. By what we have proved just above, if $x_0\in \subd\F^*(y_0)$,
then $x_0 \in X_0$.   Furthermore, by the cases of equality in Young's inequality,
\begin{equation}\label{in10}
\langle x_0,y_0\rangle = \F(x_0) + \F^*(y_0) \ .
\end{equation}
Then since $\F(x_0) = \E(x_0)$ and $\F^*(y_0) = \E^*(y_0)$, we also have $\langle x_0,y_0\rangle = \E(x_0) + \E^*(y_0)$.
Once more, by the cases of equality in Young's inequality, 
$$y_0\in \subd\E(x_0) \in   \subd\E(X_0)\ .$$
an entirely analogous argument shows that $X_0 \subset \subd\F^*(Y_0)$. 
\end{proof}

 \subsection{Quantitative  convexity}
 
 \begin{defi}[$(d,\Phi)$-convexity]  Let $X$ be a Banach space, and let $\E$ be a proper, closed convex function on $X$. 
 Let $\Phi$ be a rate function, and let $d$ be a metric on $X$.
 Let $U\subset X$ be a convex subset of the domain of $\E$. 
Then  $\E$  is  {\em $(d,\Phi)$-convex on  $U$} in case for all $x_1,x_2\in U$ and all $y\in 
\subd\E(x_1)$,
\begin{equation}\label{ams1}
\E(x_2) \geq \E(x_1) + \langle x_2-x_1,y\rangle + \Phi(d(x_1,x_2))\ .
\end{equation}
In case $d$ is the metric induced by the norm, and $\Phi(t) = \lambda t^2$ for all $t$, so that for all $x_1,x_2\in U$
\begin{equation}\label{ams2}
\E(x_2) \geq \E(x_1) + \langle x_2-x_1,y\rangle + \lambda \|x_1 - x_2\|_X^2\ ,
\end{equation}
we say that $\E$ is {\em $\lambda$-convex on $U$}. 
\end{defi}

To see that this is a strengthened form of convexity, let $z_1,z_2\in U$ and let $\alpha\in (0,1)$. Define
$x_1 = (1-\alpha)z_1+\alpha z_2$, and let $y\in \subd(\E)(x_1)$.
  Applying (\ref{ams1}) with this choice of $x_1$, and then $x_1 = z_2$ and $z_1$ in turn, we obtain
\begin{eqnarray}
\alpha \E(z_2) &\geq& \E(x_1) + \alpha\langle z_2-x_1,y\rangle + \alpha\Phi(d(x_1,z_2))\nonumber\\
(1-\alpha) \E(z_1) &\geq& \E(x_1) + (1-\alpha)\langle z_1-x_1,y\rangle + (1-\alpha)\Phi(d(x_1,z_1))\ .\nonumber
\end{eqnarray}

Adding these inequalities we obtain
\begin{eqnarray}\label{ams3}
(1-\alpha) \E(z_1)  + \alpha \E(z_2) &\geq& \E((1-\alpha)z_1+\alpha z_2) + \Phi ((1-\alpha)d(z_1,x_1)+\alpha d(x_1,z_2))\nonumber\\
&\geq& \E((1-\alpha)z_1+\alpha z_2) + \Phi(\min\{\alpha,1-\alpha\}d(z_1,z_2))\ .\nonumber
\end{eqnarray}
In the case $d$ is the metric induced by the norm on $X$ and $\Phi(t) = \lambda t^2$ we can avoid the final inequality in 
(\ref{ams3}) and compute
\begin{eqnarray*}
\Phi ((1-\alpha)d(z_1,x_1)+\alpha d(x_1,z_2)) &=& \lambda ( (1-\alpha)\|z_1 - x_1\|_X + \alpha \|Z_2 - x_1\|_X)^2\nonumber\\
&=& \lambda  \alpha(1-\alpha) \|z_1 - z_2\|_X^2 \nonumber
\end{eqnarray*}
Hence, when $\E$ is $\lambda$-convex on $U$, then 
for all $x_1,x_2\in U$ and all $\alpha\in (0,1)$,
\begin{equation}\label{lcan}
\alpha \E(x_2) + (1-\alpha)\E(x_1) \geq \E(\alpha x_2 + (1-\alpha) x_1) + \lambda \alpha(1- \alpha)\|x_2 - x_1\|_X^2\ .
\end{equation}
This is the standard definition of $\lambda$-convexity. To see that the two formulations are equivalent, 
rewrite (\ref{lcan}) as
$$\alpha(\E(x_2) - \E(x_1))  \geq \E(x_1 +\alpha (x_2 - x_1)) - \E(x_1) +  \lambda \alpha(1- \alpha)\|x_2 - x_1\|_X^2\ .$$
For $y\in \subd \E(x_1)$, $\E(x_1 +\alpha (x_2 - x_1)) - \E(x_1) \geq \alpha\langle x_2-x_2,y\rangle$, and we obtain
$$\alpha(\E(x_2) - \E(x_1))  \geq \alpha\langle x_2-x_1,y\rangle +  \lambda \alpha(1- \alpha)\|x_2 - x_1\|_X^2\ .$$

Dividing by $\alpha$, and then taking the limit $\alpha \to 0$, we obtain
\begin{equation}\label{lcan2}
y \in \subd\E(x_1) \quad \Rightarrow \quad  \E(x_2) \geq \E( x_1) + \langle x_2- x_1,y\rangle +
\lambda \|x_2 - x_1\|_X^2\ .
\end{equation}

\begin{thm}[$(d,\Phi)$-convexity and the strong Young's inequality]\label{syoung} Let $(X,Y,\langle \cdot,\cdot\rangle)$ be a dual pair. Let $\E$ be a closed convex function on $X$.  Let $d$ be a metric on 
$X$, and let $\Phi$ be a rate function. Suppose that  $\E$ is $(d,\Phi)$-convex on a convex set $U \subset X$.
Then  for all $y\in \subd\E(U)$, $\subd\E^*(y)\cap U$ is a singleton, and defining  $\nabla \E^*(y)$ to be the single element in
$\subd\E^*(y)\cap U$, 
\begin{equation}\label{lcan44}
 \E(x) + \E^*(y) \geq  \langle x,y\rangle +
\Phi(d(x ,\nabla \E^*(y)) \ 
\end{equation} is satisfied for all $x\in U$ and all $y\in \subd\E(U)$.

Conversely, if $\subd\E^*(y)\cap U$ is a singleton $\{\nabla \E^*(y)\}$ for all $y\in \subd\E(U)$ and (\ref{lcan44})
is valid for all $x\in U$ and all $y\in \subd\E(U)$, then $\E$ is $(d,\Phi)$-convex on $U$
\end{thm}

\begin{remark} We refer to inequality (\ref{lcan44}) as the {\em strong Young's inequality} for the $(d,\Phi)$-convex function
$\E$ on $U$. Theorem~\ref{syoung} says that the validity of this  inequality, together with the property that for all $y\in \subd\E(U)$, $\subd\E^*(y)\cap U$ is a singleton, characterizes $(d,\Phi)$-convexity.
\end{remark}

\begin{proof}[Proof of Theorem~\ref{syoung}]  
Let $x_1,x_2\in U$, and let $y \in \subd\E(x_1)$.
By the cases of equality in Young's inequality
$$\E(x_1) - \langle x_1,y\rangle = -\E^*(y)$$
Combining this with (\ref{ams1}) yields 
${\displaystyle 
\E(x_2) + \E^*(y) \geq  \langle x_2,y\rangle +
\Phi(d(x_2 ,x_1))}$.
Since $y \in \subd\E(x_1)  \iff x_1 \in \subd\E^*(y) $, we conclude that for all $x\in U$ and $y\in \subd\E(U)$, 
\begin{equation}\label{ams15}
\E(x) + \E^*(y) \geq  \langle x,y\rangle +
 \sup_{z\in \subd\E^*(y)\cap U} \ \Phi(d(x,z)) \ .
\end{equation}
By the ordinary Young's inequality, when $y\in \subd\E(x)$, this reduces to 
$$0 \geq   \sup_{z\in \subd\E^*(y)\cap U} \ \Phi(d(x,z))\ .$$
Evidently this implies that $\subd\E^*(y)\cap U$ is a singleton for all $y\in \subd \E(U)$.  Thus we may define a function
$y \mapsto \nabla \E^*(y)$ on $\subd\E(U)$ by defining $\nabla \E^*(y)$ to be the single element in $\subd\E^*(y)\cap U$.
Then (\ref{ams15}) simplifies to become (\ref{lcan44})

Conversely, let $x,x_1\in U$ and $y\in \subd\E(x_1)$. then $x_1 \in \subd\E^*(y)\cap U$, and hence if this set is the singleton
$\{\nabla\E^*(y)\}$, $x_1 = \nabla   \nabla\E^*(y)$. Thus, when $\subd\E^*(y)\cap U$ is a singleton $\{\nabla\E^*(y)\}$
for all $y\in \subd\E(U)$, and (\ref{lcan44}) is valid for all  $x\in U$ and $y\in \E(U)$, 
then 
$$
\E(x) + \E^*(y) \geq  \langle x,y\rangle +
\Phi(d(x ,x_1)\ .
$$
Since $\E^*(y) = \langle x_1,y\rangle - \E(x_1)$, this implies
$$
\E(x) \geq \E(x_1) \langle x-x_1,y\rangle +
\Phi(d(x ,x_1)\ .
$$
so that $\E$ is $(d,\Phi)$ convex on $U$.
\end{proof}

The next theorem, which is proved in an appendix, provides an important class of examples:

\begin{thm}[Strong convexity of squared $L^p$ norms]\label{sconlp} 
Let $1 < p  < \infty$ and let $p' = p/(p-1)$. Let $X = L^p(\Omega,\mathcal{B},\mu)$ and 
$Y = L^{p'}(\Omega,\mathcal{B},\mu)$, For $f\in X$ and $g\in Y$ let 
${\displaystyle \langle f,g\rangle = 2\Re\left( \int_{\Omega} f^* g {\rm d}\mu\right)}$, so that $X,Y$
equipped with $\langle \cdot,\cdot\rangle$ is a dual pair. Define $\E$ on $X$ by $\E(f) = \|f\|_p^2$.  Then $\E$ is differentiable, and
we have:

\smallskip
\noindent{\it (1)} For $p\in (1,2]$ and all $f_1,f_2\in X$,
\begin{equation}\label{lp2cA}
\E(f_2) \geq  \E(f_1)  + \langle f_2-f_1,\nabla\E(f_1)\rangle +  (p-1)\|f_2-f_1\|_p^2\ ,
\end{equation}
and hence for such $p$, $\E$  is $(p-1)$-convex. 

\smallskip
\noindent{\it (2)} For $p\in (2,\infty)$, $\E$ is not $\lambda$-convex for any $\lambda>0$, however,
for all $f_1,f_2\in X$,
\begin{equation}\label{lp2cA}
\E(f_2) \geq  \E(f_1)  + \langle f_2-f_1,\nabla\E(f_1)\rangle +  \frac{1}{4p}\left(\frac23\right)^{p-1}(\|f_1\|_p + \|f_2\|_p)^{2-p}\|f_2-f_1\|_p^p\ .
\end{equation} 
\end{thm}

It is well-known that some form of strong convexity of a convex function $\E$ implies certain regularity properties of the Legendre transform $\E^*$.  The next lemma is a precise version of this in a context that is relevant here. This is followed by a general result. 

\begin{thm}[Continuity of the gradient maps in $L^p$]\label{gradconlp} Let $1 < p < \infty$, and let $p' = p/(p-1)$.
Let $(\Omega, \mathcal{B},\mu)$ be a measure space, and let $X = L^p(\Omega, \mathcal{B},\mu)$ and
$Y = L^{p'}(\Omega, \mathcal{B},\mu)$  Let $\langle \cdot,\cdot\rangle$ be the bilinear form on $X\times Y$
given by $\langle f,g\rangle = 2\Re\left(\int_\Omega f^*g{\rm d}\mu\right)$.  Define the functional $\E$ on $X$
by $\E(f) = \|f\|_X^2$. Then $\E$ and $\E^*$ are differentiable, so that $\nabla \E$ and  $\nabla \E^*$ 
are well defined. 

\smallskip
\noindent{\it (1)} For $1< p \leq 2$, and all $g_1,g_2\in Y$,
\begin{equation}\label{smokra1}
\|\nabla \E^*(g_1) - \nabla \E^*(g_2) \|_X \leq \frac{1}{p-1}\|g_1-g_2\|_Y 
\end{equation}
and 

\smallskip
\noindent{\it (2)} For $2< p  < \infty$, and all $g_1,g_2\in Y$ with $\max\{ \|g_1\|_Y\ ,\ \|g_2\|_Y\} \leq R/2$,
\begin{equation}\label{smokra2}
\|\nabla \E^*(g_1) - \nabla \E^*(g_2) \|_X \leq \left[\frac23\left(\frac{1}{4p}\right)^{1/(p-1)}
R^{(2-p)/(p-1)}\right] \|g_1-g_2\|_Y^{1/(p-1)}\ . 
\end{equation}
\end{thm}

\begin{proof}  The differentiability of  $\E$ and $\E^*$ is well-known.
Suppose first that $1 < p \leq2$.  Let $g_1,g_2\in Y$ and define $f_j = \nabla \E^*(g_j)$, $j = 1,2$. 
Then $g_j = \nabla \E(f_j)$, $j = 1,2$.
By Theorem~\ref{sconlp},
$$
\E(f_2) - \E(f_1) \geq \langle f_2-f_1,g_1\rangle + (p-1)d_X(f_1,f_2)$$
and
$$
\E(f_1) - \E(f_2) \geq \langle f_1-f_2,g_2\rangle + (p-1) d_X(f_1,f_2)\ .
$$
Summing these inequalities,
$$0 \geq   \langle f_2-f_1,g_1-g_2\rangle + 2(p-1) d_X(f_1,f_2)\ .$$
Then since
$2\|f_2- f_1\|_X \|g_2- g_1\|_Y \geq  \langle f_2-f_1,g_1-g_2\rangle$ we conclude
$$\|f_2- f_1\|_X \|g_2- g_1\|_Y  \geq (p-1)\|f_1- f_2\|_X^2\ .$$
This proves (\ref{smokra1}).

Next, suppose that $2 < p < \infty$.  Let $g_1,g_2\in Y$ and define $f_j = \nabla \E^*(g_j)$, $j = 1,2$. 
Then $g_j = \nabla \E(f_j)$, $j = 1,2$.
By Theorem~\ref{sconlp},
$$
\E(f_2) - \E(f_1) \geq \langle f_2-f_1,g_1\rangle + \frac{1}{4p}\left(\frac23\right)^{p-1}(\|f_1\|_p + \|f_2\|_p)^{2-p}\|f_2-f_1\|_p^p$$
and
$$
\E(f_1) - \E(f_2) \geq \langle f_1-f_2,g_2\rangle + \frac{1}{4p}\left(\frac23\right)^{p-1}(\|f_1\|_p + \|f_2\|_p)^{2-p}\|f_2-f_1\|_p^p\ .
$$
Summing these inequalities,
$$0 \geq   \langle f_2-f_1,g_1-g_2\rangle + \frac{1}{2p}\left(\frac23\right)^{p-1}(\|f_1\|_p + \|f_2\|_p)^{2-p}\|f_2-f_1\|_p^p\ .$$
Proceeding as above, we conclude
$$
\|g_1-g_2\|_Y \geq \frac{1}{4p}\left(\frac23\right)^{p-1}(\|f_1\|_p + \|f_2\|_p)^{2-p}\|f_2-f_1\|_p^{p-1}\ .
$$
Therefore, using the norm preserving property $\|f_j\|_X = \|g_j\|_Y$, $j=1,2$, 
$$\|\nabla\E^*(g_2)-\nabla\E^*(g_1)\|_p \leq  \left[\frac23\left(\frac{1}{4p}\right)^{1/(p-1)}
(\|g_1\|_p + \|g_2\|_{p'})^{(2-p)/(p-1)}\right] \|g_2-g_1\|_{p'}^{1/(p-1)}\ .$$

\end{proof}

The next lemma is a general result asserting that quantitative convexity of a functional implies quantitative smoothness of its Legendre transform.

\begin{lm}[$(d,\Phi)$-convexity of $\E$ and regularity of $\E^*$]\label{dureg2} 
Let $X$ and $Y$ be a dual pair of Banach spaces. Let $\E$ be a proper lower semicontinuous convex function on  $X$. Let $U$ be an open convex subset of $X$. Suppose that
$\E$ is $(d,\Phi)$-convex on $U$ for some metric $d$ on $X$ and some rate function $\Phi$.  Then 
 $\subd\E^*(y)$ is a singleton $\{\nabla \E^*(y)\}$, and 
 for all $y_1,y_2\in \subd\E(U)$,
\begin{equation}\label{lipdu}
\frac{\Phi(d( \E^*(y_2),\E^*(y_1)  )}{\|\nabla \E^*(y_2) - \nabla \E^*(y_1)\|_X} \leq \|y_2- y_1\|_Y\ .
\end{equation}
\end{lm}

\begin{proof}We have already seen that the map $y \mapsto \nabla \E^*(y)$ is well-defined. We first show that (\ref{lipdu})
is satisfied for all $y_1,y_2\in \subd\E(U)$.

Let $y_1,y_2\in  \subd\E(U)$. Let $x_j = \nabla \E^*(y_j)$, $j=1,2$. Then $y_j \in \subd\E(x_j)$, $j=1,2$,
and hence
$$
\E(x_2) - \E(x_1) \geq \langle x_2-x_1,y_1\rangle + \Phi(d(x_1,x_2))$$
and
$$
\E(x_1) - \E(x_2) \geq \langle x_1-x_2,y_2\rangle + \Phi(d(x_1,x_2))\ .
$$
Adding these inequalities and rearranging terms,
\begin{equation}\label{ams35}
2\Phi(d(x_1,x_2)) \leq \langle x_2-x_1,y_2-y_1\rangle \leq 2\|x_2 -x_1\|_X\|y_2- y_1\|_Y\ ,
\end{equation}
which implies  (\ref{lipdu}). 

Now suppose that $U$ is open and $\E$ itself is differentiable everywhere on $U$, so that $x\mapsto \nabla \E(x)$ is a well defined map on $U$.
We have seen in Lemma~\ref{youngin} that for any proper lower semicontinuous convex function $\E$, 
$y\in \subd\E(x)$ which is the case if and only if   $x\in \subd\E^*(y)$.
Since $y = \nabla\E(x)$ is in $\subd\E(x)$, $x\in \subd\E^*(y) = \{\nabla \E^*(y)\}$. That is, $x= \nabla \E^*(y) = \nabla \E^*(\nabla\E(x))$. Combining $y = \nabla\E(x)$ with $x= \nabla \E^*(y)$, we also obtain $y= \nabla\E(\nabla \E^*(y))$. 
\end{proof}

As indicated by Theorem~\ref {gradconlp}, the inequality (\ref{lipdu})  simplifies when $d$ is the metric induced by the norm on $X$. Define the function
$\Psi$  in terms of the rate function $\Phi$ by
\begin{equation}\label{psidef}
\Psi(t) = \frac{\Phi(t)}{t}\ 
\end{equation}
for all $t>0$. Since for any strictly convex function $\varphi$ on $\R$, 
$$t_2 > t_1  > t_0 \quad \rightarrow \quad \frac{\varphi(t_2) - \varphi(t_0)}{t_2 - t_0} \geq 
\frac{\varphi(t_1) - \varphi(t_0)}{t_1 - t_0} \ ,$$
$\Psi$ is a strictly increasing function on $(0,\infty)$. Suppose that
\begin{equation}\label{psireg}
\lim_{t\downarrow 0}\Psi(t) = 0\ .
\end{equation}

Using this notation, 
(\ref{lipdu})  becomes 
\begin{equation}\label{lipduB}
\Psi( \|\nabla \E^*(y_2) - \nabla \E^*(y_1)\|_X) \leq \|y_2- y_1\|_Y\ .
\end{equation}
Moreover, under the condition (\ref{psireg}), (\ref{lipduB}) implies that
 that $y\mapsto \nabla \E^*(y)$ is  continuous on $\subd\E(U)$.
 Hence, when $\E$ is differentiable, so that $y \mapsto \nabla \E^*(y)$ is invertible, $x \mapsto \nabla\E$ is open on $U$,
 and so $\nabla\E(U)$ is open.  In this case, $y \mapsto \nabla \E^*(y)$ is differentiable on $\nabla\E(U)$.
 
 To see this, let $y_1,y_2\in \nabla\E(U)$. Then there are uniquely determined $x_1,x_2 \in U$  such that 
 $x_j = \nabla \E^*(y_j)$, $j=1,2$. Therefore,
\begin{equation}\label{lipdu3}
\E^*(y_2) \geq \E^*(y_1) + \langle x_1,y_2-y_1\rangle \qquad{\rm and}\qquad 
\E^*(y_1) \geq \E^*(y_2) + \langle x_2,y_1-y_2\rangle\ .
\end{equation}
Since
$$ \langle x_2,y_1-y_2\rangle  = \langle x_1,y_1-y_2\rangle + \langle x_2-x_1,y_1-y_2\rangle\  ,$$
the inequality (\ref{lipduB}) says that
$$|\E^*(y_2) - \E^*(y_1) - \langle x_1,y_2-y_1\rangle| \leq \|y_1-y_1\|_Y\Psi( \|y_1-y_1\|_Y) \ ,$$
which shows that $\E^*$ is differentiable at $y_1$, and the derivative is $x_1 = \nabla \E^*(y_1)$, which finally justifies our notation.

\section{Duality and stability}

When {\em either} $\E^*$ or $\E^*$ is $(d,\Phi)$ convex on their respective domains, we may use the strengthened form of Young's inequality
(\ref{lcan44}) to strengthen the conclusion of Lemma~\ref{reverse}.

\begin{lm}\label{reverse2} Let $\E$ and $\F$ be proper, lower semicontinuous convex functions on $X$.  

\smallskip
\noindent{\it(1)} 
If  for some metric $d_Y$ on $Y$, and some rate function $\Phi_Y$,  $\E^*$ is $(d_Y,\Phi_Y)$-convex, then for 
all $y$ in the domain if $\F^*$ and all $x \in \subd\F^*(y)$, 
\begin{equation}\label{rev51}  
\E^*(y) - \F^*(y) \geq \F(x) - \E(x) + \Phi_Y(d_Y(y,\nabla\E(x)))\ .
\end{equation}

\smallskip
\noindent{\it(2)} 
If  for some metric $d_X$ on $X$, and some rate function $\Phi_X$,  $\E$ is $(d_X,\Phi_X)$-convex, then for 
all $x$ in the domain of $\F$ and all
$y \in \subd\F(x)$, 
\begin{equation}\label{rev51B}  
\E^*(y) - \F^*(y) \geq \F(x) - \E(x) + \Phi_Y(d_Y(y,\nabla\E(x)))\ .
\end{equation}
\end{lm}

\begin{proof} We first prove {\it (1)}. Suppose that $\E^*$ is $(d_Y,\Phi_Y)$-convex.  Since $x\in \subd\F^*(y)$, there is equality in Young's inequality so that
\begin{equation}\label{kra1}
\F(x) + \F^*(y) = \langle x,y\rangle\ .   
\end{equation}
Then by the $(d_Y,\Phi_Y)$-convexity of $\E^*$, expressed in the form (\ref{lcan44}),
$$\langle x,y\rangle \leq  \E(x) + \E^*(y) - \Phi_Y(d_Y(y,\nabla\E(x)))\  .$$
Combining this with the identity (\ref{kra1}), we obtain
$$
\F(x) + \F^*(y) \leq \E(x) + \E^*(y) - \Phi_Y(d_Y(y,\nabla\E(x)))\ .
$$
Rearranging terms we obtain (\ref{rev51}).  The proof of (\ref{rev51B}) is entirely analogous. 
\end{proof}

Now suppose that $\E$ and $\F$ are proper, closed convex functions, and that the functional inequality $\E \leq \F$ is not only valid, but is $(d_X,\Phi_X)$ stable.  Suppose further that $\E^*$ is $(d_Y,\Phi_Y)$-convex.  Then by {\it (1)} of Lemma~\ref{reverse2}, for all $y$ in the domain of $\F^*$, 
\begin{equation}\label{rev55}  
\E^*(y) - \F^*(y) \geq \Phi_X(d_X(x, X_0)) + \Phi_Y(d_Y(y,\nabla\E(x)))\ .
\end{equation}

For the right hand side to be small, we must have $y\approx \nabla \E(x)$ and $x\approx x_0 \in X_0$.
The strong convexity of $\E^*$ implies that $\nabla \E$ is continuous, and so
$x\approx x_0$ implies $\nabla \E(x) \approx \nabla \E(x_0)$. Then $y\approx \nabla \E(x_0)$, and since
$\nabla \E(x_0)\in Y_0$, $y$ is necessarily close to an optimizer for the dual inequality
$\F^* \leq \E^*$. 

As shown below, going though this line of argument with precise estimates ``transfers'' a stability result for 
$\E \leq \F$ to its dual $\F^* \leq \E^*$.

As we have seen in the case of squared $L^p$ norms, it may be the $\E$ has better quantitative convexity properties than $\E$. 
In this case, we may combine the $(d_X,\Phi_X)$ stability of $\E \leq \F$ with  $(d_X,\Phi_X)$-convexity of $\E$ using {\it (2)} of 
Lemma~\ref{reverse2} to obtain
\begin{equation}\label{rev55}  
\E^*(y) - \F^*(y) \geq \Phi_X(d_X(x, X_0)) + \Phi_X(d_X(x,\nabla\E^*(y)))\ .
\end{equation}
The sort of reasoning described just above again leads to the conclusion that if $\E^*(y) \approx \F^*(y)$, then $\nabla \E^*(y) \approx x_0\in X_0$. But since $X_0 = \nabla \E^*(Y_0)$, this means that 
a stability bound for $\F^* \leq \E^*$, and even more simply since now we are comparing distances in the same space.  
In the next subsection, we carry out a general development of this idea.  First however, we use Lemma~\ref{reverse2} to prove Theorem~\ref{HLSB3}

\begin{proof}[Proof of Theorem~\ref{HLSB3}]  We use the notation from the introduction:
\begin{equation}\label{tnorm1X}
\|f\|_X = \|f\|_{2n/(n-2\alpha)} \qquad{\rm and}\quad  \tnorm f\tnorm_X = 
\|(-\Delta)^{\alpha/2} f\|_2\ ,
\end{equation}
and 
\begin{equation}\label{tnorm2Y}
\|g\|_Y = \|f\|_{2n/(n+2\alpha)} \qquad{\rm and}\quad  \tnorm g\tnorm_Y = 
\|(-\Delta)^{-\alpha/2} g\|_2\ .
\end{equation}
Let $\F(f) = \S_{n,\alpha}\tnorm f\tnorm_X^2$ and $\E(f) = \|f\|_X^2$. 
With respect to the bilinear form 
$\langle f,g\rangle  = 2\Re \left(\int_{\R^n} f^* g {\rm d}m\right)$, 
$\F^*(g) = \S_{n,\alpha}^{-1}\tnorm g\tnorm_Y^2$ and $\E^*(g)  = \|g\|_Y^2$.

Recall that  general version of the Bianchi-Egnell Theorem proved \cite{CFW} states that there exists a  constant $\kappa_{BE}>0$ depending only on $n$ and $\alpha$ such that 
\begin{equation}\label{aakra5}
\F(f) - \E(f) \geq \kappa_{BE} \inf_{z\in X_0}\{\tnorm f-z\tnorm_X^2\}   \ .
\end{equation}
That is, the Sobolev inequality is $\kappa_{BE}$ stable. 

Since $\frac{2n}{n+2\alpha} < 2$ and $\frac{2n}{n+2\alpha} -1 = \frac{n-2\alpha}{n+2\alpha}$, Theorem~\ref{sconlp} says that $\E^*$ is 
 $\frac{n-2\alpha}{n+2\alpha}$-convex. Therefore, we have what we need to apply {\it (1)} of Lemma~\ref{reverse2}. 
 Doing so we obtain that for $f = \nabla \F^*(g)$, 
 \begin{equation}\label{iml1}
  \E^*(g) - \F^*(g) \geq \F(f) - \E(f) + \frac{n-2\alpha}{n+2\alpha}\| g - \nabla \E(f)\|_Y^2\ .
 \end{equation}

By the Sobolev inequality itself, for all $f_0\in X_0$, 
\begin{equation*}
\tnorm f- f_0\tnorm_X \geq \S_{n,\alpha}^{-1/2}\|f - f_0\|_X  \geq 
 \S_{n,\alpha}^{-1/2}\inf_{z\in X_0}\{\|f-z\|_X\}\  .
 \end{equation*}
 Therefore,
 $$
 \inf_{z\in X_0}\{\tnorm f-z\tnorm_X^2\}  \geq 
 \S_{n,\alpha}^{-1/2}\inf_{z\in X_0}\{\|f-z\|_X\}\  .
$$
Combining this with (\ref{aakra5}) and (\ref{iml1}) yields
 \begin{equation*}
 \E^*(y) - \F^*(y) \geq \S_{n,\alpha}^{-1}\kappa_{BE} \inf_{z\in X_0}\{\| f-z\|_X^2\}   +
 \frac{n-2\alpha}{n+2\alpha}\| g - \nabla \E(f)\|_Y^2
 \end{equation*}
 Hence, for all $\epsilon>0$, there is an $f_0\in X_0$ such that
 \begin{equation}\label{iml2}
 \E^*(y) - \F^*(y) \geq \S_{n,\alpha}^{-1}\kappa_{BE} \| f-z\|_X^2  + 
 \frac{n-2\alpha}{n+2\alpha}\| g - \nabla \E(f)\|_Y^2 - \epsilon\ .
 \end{equation}
 By Theorem~\ref{gradconlp},  since $\E^*$ is 
 $\frac{n-2\alpha}{n+2\alpha}$-convex, $\nabla \E$ is Lipschitz from $X$ to $Y$
 with Lipschitz constant $\frac{n+2\alpha}{n-2\alpha}$. Hence
 $$\|f-f_0\|_X \geq  \frac{n-2\alpha}{n+2\alpha}\|\nabla \E(f) - \nabla \E(f_0)\|_Y\ .$$
 Using this in (\ref{iml2}) yields:
 \begin{eqnarray*}
 \E^*(y) - \F^*(y) &\geq&  \frac{n-2\alpha}{n+2\alpha}\min\left\{ \S_{n,\alpha}^{-1}(n)\kappa_{BE}  \frac{n-2\alpha}{n+2\alpha}\ ,\ 1 \right\}
 \left( \|f-f_0\|_X^2 +  \| g - \nabla \E(f)\|_Y^2\right)- \epsilon\\
 &\geq& \frac{n-2\alpha}{n+2\alpha}\min\left\{ \S_{n,\alpha}^{-1}(n)\kappa_{BE} \frac{n-2\alpha}{n+2\alpha}\ ,\ 1 \right\}  \frac12  \left( 
 \|f-f_0\|_X +  \| g - \nabla \E(f)\|_Y   \right)^2 -\epsilon \ .
 \end{eqnarray*}
 
 Therefore, by the triangle inequality and the fact that $ \nabla \E(f_0) \in Y_0$,
 \begin{eqnarray}\label{aakra2}
  \|f-f_0\|_X +   \| g - \nabla \E(f)\|_Y &\geq&
   \left(\|\nabla \E(f) - \nabla \E(f_0)\|_Y +  \| g - \nabla \E(f)\|_Y\right)\nonumber\\
   &\geq&  \left(\|g - \nabla \E(f_0)\|_Y \right)
   \geq  \inf_{w\in Y_0}\{ \|g - w\|_Y\}\ .\nonumber
   \end{eqnarray}
   Noting that $\epsilon$ can be chosen arbitrarily small, we obtain the result with
 \begin{equation}\label{kbestar}
 \kappa^*_{BE} = \frac12 \frac{n-2\alpha}{n+2\alpha}\min\left\{ \S_{n,\alpha}^{-1}(n)\kappa_{BE}  \frac{n-2\alpha}{n+2\alpha}\ ,\ 1 \right\}  \ .
 \end{equation}
   \end{proof}

\subsection{Duality  for stability inequalities in general}

In this subsection we return to  the general setting of an abstract dual pair $X,Y$. Suppose that  for some metric $d_X$ on $X$, and some rate function $\Phi_X$, the inequality $\E \leq \F$ is 
$(d_X,\Phi_X)$ stable, and for some rate function $\Phi_Y$ on $Y$, $\E^*$ is $(\Phi_Y,d_Y)$ stable. 

Let $\Phi$ be any rate function  such that
$\Phi(t) \leq \min\{\Phi_X(t)\, \ \Phi_Y(t)\ \}$
for all $t$. We have seen  that such a $\Phi$ always exists. Thus it is no loss of generality to 
suppose that $\Phi_X = \Phi_Y = \Phi$. Then, by the convexity of $\Phi$, for any $a,b \geq 0$, 
$\Phi(a) + \Phi(b) \geq 2\Phi\left((a+b)/2\right).$
Thus, (\ref{rev55}) implies  that for all $y\in Y$ and $x= \nabla\F^*(y)$, 
\begin{equation}\label{rev55Z}  
\E^*(y) - \F^*(y) \geq 2\Phi\left( \frac12 d_X(x, X_0) + \frac12 d_Y(y,\nabla\E(x))\right)\ .
\end{equation}

Note that if $x\in X_0$, then  $y\in Y_0$, and we may assume that this is not the case. 
 For any $0< \epsilon < d_X(x,X_0)$, we can find $x_0\in X_0$ such that $d_X(x,X_0) \geq d_X(x,x_0)-  \epsilon$, and hence
\begin{equation}\label{rev55Z}  
\E^*(y) - \F^*(y) \geq 2\Phi\left( \frac12 d_X(x, x_0) + \frac12 d_Y(y,\nabla\E(x)) - \epsilon\right)^2\ .
\end{equation}

We have seen in Lemma~\ref{dureg2} and the remarks following it that if  $d_X$ and $d_Y$ are the metrics induced by the norms, and if
$\E^*$ is $\kappa$-convex,  $\nabla \E$ is Lipschitz. In our present more general setting, a variety of assumptions
on the metric $d_Y$ and the rate function $\Phi_Y$ such that $\E^*$  is $(d_Y,\Phi_Y)$-convex imply the existence
of a rate function  $\Psi$ such that
 for all $x_1,x_2\in X$, 
 \begin{equation}\label{kra48}
d_X(x_1,x_2) \geq \Psi( d_Y(\nabla \E(x_1),\nabla \E(x_2)) )\ .
\end{equation}
Using this in (\ref{rev55Z}), we obtain
\begin{equation}\label{rev55Z7}  
\E^*(y) - \F^*(y) \geq 2\Phi\left( \frac12 \Psi( d_Y(\nabla \E(x),\nabla \E(x_0)) + \frac12 d_Y(y,\nabla\E(x)) - \epsilon\right)^2\ .
\end{equation}

By the triangle inequality and monotonicity of $\Psi$, 
$$\Psi( d_Y(\nabla \E(x),\nabla \E(x_0))  \geq \Psi( d_Y(y,\nabla \E(x_0) - d_Y(\nabla \E(x),y))))\ .$$
Therefore,
$$\Psi( d_Y(\nabla \E(x),\nabla \E(x_0)) + d_Y(y,\nabla\E(x)) \geq 
\inf_{s> 0} \{ \Psi(  d_Y(y,\nabla \E(x_0) - s) + s\}\ .$$
Define 
\begin{equation}\label{psihat}
\chi(s) = \begin{cases} s & s \geq 0\\ \infty & s < 0\end{cases}\quad{\rm and}\quad 
\widehat \Psi(t) = \inf_{s\in \R} \{ \Psi(t-s) + \chi(s)\}\ .
\end{equation}
The function $\widehat \Psi$ is the {\em infimal convolution} of $\Psi$ and $\chi$, $\Psi\box\chi$.  
Since the infimal convolution of convex functions is convex, 
$\widehat \Psi$ is convex, and evidently it is
strictly positive on $(0,\infty)$. Hence $\widehat \Psi$ is a rate function.   We have proved:

\begin{equation}\label{kra47}
d_X(x,x_0) = \|x-x_0\|_X \geq \kappa \|\nabla \E(x) - \nabla\E(x_0)\|_Y = \kappa d_Y(\nabla \E(x),\nabla \E(x_0)) \ .
\end{equation}
In this case, by the triangle inequality, 
\begin{eqnarray}
d_X(x,x_0) +   d_Y(y,\nabla\E(x))  &\geq& \kappa d_Y(\nabla \E(x),\nabla \E(x_0)) + d_Y(y,\nabla\E(x))\nonumber\\
&\geq& \min\{\kappa,1\}\left( d_Y(\nabla \E(x),\nabla \E(x_0)) + d_Y(y,\nabla\E(x))   \right) \nonumber\\
&\geq& \min\{\kappa,1\} d_Y(y,\nabla\E(x_0))\nonumber
\end{eqnarray}
By Lemma~\ref{reverse}, $\nabla\E(x_0) \in Y_0$, the set of optimizers for $\F^* \leq \E^*$. Thus,
$$d_X(x,x_0) +   d_Y(y,\nabla\E(x)) \geq  \min\{2\kappa,1\} d_Y(y,Y_0)\ .$$
Using this in (\ref{rev55Z}) and using the fact that $\epsilon>0$ is arbitrary, we finally obtain
\begin{equation}\label{rev55Z3}  
\E^*(y) - \F^*(y) \geq 2\Phi\big(\min\{2\kappa,1\} d_Y(y,Y_0) \big)\ .
\end{equation}
we have proved:

\begin{thm}\label{reverse2Zur} Let $\E$ and $\F$ be proper, lower semicontinuous convex functions on $X$.  
Suppose that the functional inequality $\E \leq \F$ is valid and optimal. Let $\Phi$ be a rate function. 
Suppose also that:

\smallskip
\noindent{\it (1)} For some metric $d_Y$ on $Y$, and some rate function $\Phi$,  $\E^*$ is $(d_Y,\Phi)$-convex. 

\smallskip
\noindent{\it (2)} For some metric $d_X$ on $X$, $\E \leq \F$ is  $(d_X,\Phi)$-stable.

Suppose that for some rate function $\Psi$,
\begin{equation}\label{kra48WZ}
d_X(x_1,x_2) \geq \Psi( d_Y(\nabla \E(x_1),\nabla \E(x_2)) )
\end{equation}
for all $x_1,x_2\in X$.
Define the function $\widehat \Psi$ on $[0,\infty)$ by $\widehat \Psi(t) = \inf_{s>0} \left\{ \Psi(t- s) + s\ \right\}$.
Then $\widehat\Psi$ is a rate function and 
for all $y\in Y$,
\begin{equation}\label{rev5}  
\E^*(y) - \F^*(y) \geq \Phi(\widehat \Psi(d_Y(y - Y_0)) )  \ .
\end{equation}
\end{thm}

There is a variant in which the lower bound is in terms of $d_X(\nabla \E^*(y),X_0)$ instead of
$d_Y(y,Y_0)$.  This has the same qualitative meaning since by Theorem~\ref{reverse},  $\nabla \E^*(y) \in X_0$ if and only if $y\in Y_0$.

To obtain this, we assume that $y\mapsto \nabla \E^*(y)$  satisfies an inequality of the form
\begin{equation}\label{lip4}
\Psi( d_X ( \E^*(y_1) , \nabla \E^*(y_2)) \leq d_Y(y_1, y_2)\ ,
\end{equation}
This will of course be the case if $\E$ has good convexity properties. This variant will be useful when $\E$ is $\lambda$ convex for some $\lambda>0$; e.g., when $\E(f) = \|f\|_p^2$ for some $1< p \leq 2$.


\begin{thm}\label{reverse2B} Let $\E$ and $\F$ be proper, lower semicontinuous convex functions on $X$.  
Suppose that the functional inequality $\E \leq \F$ is valid and optimal. Let $\Phi$ be a rate function. 
Suppose also that:

\smallskip
\noindent{\it (1)} For some metric $d_Y$ on $Y$, and some rate function $\Phi$,  $\E^*$ is $(d_Y,\Phi)$-convex. 

\smallskip
\noindent{\it (2)} For some metric $d_X$ on $X$, $\E \leq \F$ is  $(d_X,\Phi)$-stable. 

\smallskip
\noindent{\it (3)} For some rate function $\Psi$, (\ref{lip4})
is satisfied for all $y_1,y_2\in Y$.

Define the function $\widehat \Psi$ on $[0,\infty)$ by $\widehat \Psi(t) = \inf_{s\in \R} \left\{ \Psi(|t- s|) + |s|\ \right\}$.
Then $\widehat\Psi$ is a rate function and 
for all $y\in Y$
\begin{equation}\label{rev5}  
\E^*(y) - \F^*(y) \geq 2\Phi(\widehat \Psi(d_X(\nabla\E^*(y),X_0))  \ .
\end{equation}
\end{thm}

\begin{proof} Arguing as before, we see that for every $\epsilon>0$, there is an $x_0\in X_0$ such that 
\begin{equation}\label{krak2Z}
\E^*(y) - \F^*(y) \geq  2\Phi\left( \frac12 d_X(x,x_0) + \frac12  d_Y(y,\nabla\E(x))  +\epsilon\right)\ \ .
\end{equation}
Under our hypotheses, $\nabla \E$ and $\nabla \E^*$ are inverse to one another. Define $x_1 = \nabla \E^*(y)$
 so that
$y = \nabla \E(x_1)$.
 Then by (\ref{lip4}),
$$d_Y(y,\nabla\E(x))  = d_Y(\nabla\E(x_1),\nabla\E(x)) \geq 
\Psi(d_X(\nabla\E^*(\nabla\E(x_1)), \nabla\E^*(\nabla\E(x_x))) = \Psi(d_X(x_1,x))\ .$$
Using this in (\ref{krak2Z}), we obtain
\begin{eqnarray}\label{krak3}
\E^*(y) - \F^*(y) &\geq&  2\Phi\left( \frac12 \|x-x_0\|_X 
 + \frac12 \Psi  \|\nabla\E^*(y) - x\|_X) +\epsilon\right)\ \nonumber\\
 &\geq& \inf_{x_0\in Y_0} \inf_{z\in X} \left\{ 2\Phi\left( \frac12 \Psi(\| \nabla \E^*(y)  - z\|_X)
 + \frac12  \|x_0 - z\|_X +\epsilon\right)\right \} \nonumber\\
 &\geq& \inf_{x_0\in Y_0} \inf_{z\in X} \left\{ 2\Phi\left( \frac12 \Psi(\| \nabla \E^*(y)  - x_0\|_X  - \|x_0- z\|_X)
 + \frac12  \|x_0 - z\|_X\right)\right \} \nonumber\\
  &\geq& 2\Phi(\widehat \Psi(d_X(\nabla\E^*(y),X_0))) \ .\nonumber
\end{eqnarray}

\end{proof}

\if false

Before developing this argument in general, we consider the special case of the Sobolev inequality and its 
dual, the HLS inequality.  We shall obtain significant new results in doing so:  First, we shall show 
how the Bianchi-Egnell inequality ``dualizes'' to yields a new stability inequality for the HLS inequality. 

We  shall also see how a {\em local stability bound}; that is an statement  of the form
$$ d_X(x,X_0) < r\tnorm x\tnorm_X \quad \Rightarrow \quad \F(x) - \E(x) \geq \Phi(d_X(x,x_0))$$
implies a local statement for the dual inequality.   We have such an inequality for the Sobolev inequality
where the rate function $\Phi$ is of the form $\Phi(t) = \lambda t^2$ and both $r$ and $\lambda$ are 
{\em explicit} positive constant.  We shall use duality to prove a corresponding result for the HLS inequality.
We then take advantage of the a monotonicity property of the HLS inequality under a certain to obtain 
a {\em global} quantitative stability bound for the HLS inequality from the local bound. Finally, we use duality once more to obtain a {\em global} quantitative stability bound for the Sobolev inequality. 

We carry out the analysis for Sobolev and HLS inequalities first in part because of their interest, but also because
once these examples are understood, the various generalizations we introduce next will be well-motivated and easy to prove.
\fi 

\subsection{Duality and local stability}

The method used by  Bianchi and Egnell to prove (\ref{aakra5})  can be elaborated to give a {\em quantitative} stability bound for the Sobolev inequality,
but only locally, near to $X_0$: As discussed above, their local estimate comes from an eigenvalue calculation and control of remainder terms in a Taylor expansion, and no essential use of qualitative compactness arguments is made up to this point.  Though they did not obtain quantitative control over the remainder terms in the Taylor expansion, this can be done using uniform convexity inequalities; see Section 2 of \cite{seuf}.

The optimal  logarithmic HLS inequality on $\R^2$  can be obtained from the $HLS$ inequality by taking the limit $\alpha \to 1$; see \cite{CL}. It would be tempting to try and extract a stability result for the Logarithmic HLS inequality from Theorem~\ref{HLSB3} using the 
expression for $\kappa_{BE}^*$
given in (\ref{kbestar}). However, $\kappa_{BE}^*$ depends $\kappa_{BE}$, and there is no information available on how $\kappa_{BE}$ depends on $\alpha$. 

It is therefore of interest that the {\em local} Bianchi-Egnell result, which can be made quantitative,
also transfers by duality. Application of this to stability for the Logarithmic HLS inequality will be made elsewhere, but it is natural to explain here how the transfer of stability by duality localizes.

In our notation, 
the local stability result of Bianchi and Egnell says that there exist  computable 
constants $r,\lambda >0$
such that 
\begin{equation}\label{local1}
\inf_{z\in X_0}\{ \tnorm f-z\tnorm_X \} \leq r \tnorm f \tnorm_X   \quad \Rightarrow \quad  \F(f) - \E(f) 
 \geq \lambda    \inf_{z\in X_0}\{ \tnorm f-z\tnorm_X^2\} \ ,
\end{equation}
where we are using the notation introduced in (\ref{tnorm1}) and (\ref{tnorm1B}). In the rest of this section, we use the notation from  the introduction in the discussion there of the Sobolev and HLS inequalities, as in the proof of Theorem~\ref{HLSB3}.
(In \cite{BE}, (\ref{local1}) was proved for $\alpha =1$ and any $r < 1$, but with a $o(\inf_{z\in X_0}\{ \tnorm f-z\tnorm_X^2 \})$ Taylor expansion remainder term on the right. Decreasing $r$ further, one may absorb the Taylor remainder term into the main term since $\lambda > 0$. See also Theorem 3 in \cite{CFW} for general $\alpha$.)

We can express (\ref{local1}) as saying that the inequality $\E \leq \F$ has is {\em locally $\lambda$-stable with respect to the $\tnorm \cdot \tnorm_X$
metric}.   We now show that this local stability result  ``dualizes'' to yield a local quantitative  stability result for
the HLS inequality. We do this for $\alpha =1$ only for simplicity, and write $\S_n$ for $\S_{n,1}$.

\begin{lm}[Local quantitative local stability for the HLS inequality]\label{qloc}
Let $\lambda,r>0$ be such that the local Bianchi-Egnell inequality (\ref{local1}) is valid. 
Then for all $g\in Y= L^{2n/(n+2)}$ such that $2\F^*(g) \geq  \E^*(g)$, 
\begin{equation}\label{local1d}
d_Y(g,Y_0) \leq \frac{r}{2}\|g\|_Y \quad \Rightarrow 
\E^*(g) - \F^*(g) \geq \frac12 \frac{n-2}{n+2} \min\left\{
 \lambda\frac{4}{\S_n} \frac{n-2}{n+2}\ ,\ 1\right\}  
d_Y^2(g,Y_0)\ . 
\end{equation}
\end{lm}

\begin{proof}
It is  easy to see that  $\nabla \F^*(g) = (-\Delta)^{-1}g$, and consequently, for all $g_1,g_2\in Y$,
\begin{equation}\label{local1b}
\tnorm \nabla \F^*(g_1)  - \nabla \F^*(g_2)\tnorm_X \leq \tnorm g_1 -g_2\tnorm_Y \leq \sqrt{\S_n }\|g_1-g_2\|_Y\ ,
\end{equation}
where the last inequality is the HLS inequality. That is,
$$\inf_{z\in X_0} \{ \tnorm \nabla \F^*(g)-z\tnorm_X \}   \leq \sqrt{\S_n}\inf_{w\in Y_0}\{ \| g-w\|_Y \} \ .$$
Next, since $2\F^*(g) \geq  \E^*(g)$,
$$\S_n \|g\|_Y^2 \leq 2\|(-\Delta)^{-1/2}g\|_2^2 = 2\|(-\Delta)^{1/2}[(-\Delta)^{-1}g]\|_2^2 =
2\|(-\Delta)^{1/2}[\nabla \F^*(g)]\|_2^2 $$

 Next, since $2\F^*(g) \geq  \E^*(g)$,
It follows from (\ref{local1}) and (\ref{local1b})  that
\begin{equation}\label{local1c}
\inf_{w\in Y_0}\{ \| g-w\|_Y \} \leq \frac{r}{\sqrt{2}}\|g\|_Y \quad \Rightarrow \quad  
\inf_{z\in X_0} \{ \tnorm \nabla \F^*(g)-z\tnorm_X \} \leq r\tnorm \nabla \F^*(g)\tnorm_X\ .
\end{equation}
and then as a consequence of (\ref{local1}),
\begin{equation}\label{local1d}
\inf_{w\in Y_0}\{ \| g-w\|_Y \} \leq \frac{r}{\sqrt{2 }} \quad \Rightarrow \quad  
\F(\nabla \F^*(g)) - \E(\nabla \F^*(g))  \geq \lambda  \left(\inf_{z\in X_0}\{ \tnorm \nabla \F^*(g)-z\tnorm_X \}\right)^2 \ . 
\end{equation}
Since  $\E^*$ is $\frac{n-2}{n+2}$-convex,  Part {\it (1)} of Lemma~\ref{reverse2}  says that whenever  
$\inf_{w\in Y_0}\{ \| g-w\|_Y \} \leq \frac{r}{\sqrt{2 }}\tnorm g\tnorm_Y$, for all $\epsilon>0$, there is an $f_0\in X_0$ such that with $f = \nabla \F^*(g)$, 
\begin{equation}\label{local1e}
\E^*(g) - \F^*(g) 
\geq \lambda \tnorm f-f_0\tnorm_X^2 - \epsilon+ \frac{n-2}{n+2}\|g - \nabla \E(f)\|_Y^2
\end{equation}

Since $\E^*$ is $\frac{n-2}{n+2}$-convex, $\nabla \E$ is  Lipschitz from $L^{2n/(n-2)}(\R^n)$ to 
$L^{2n/(n+2)}(\R^n)$ with Lipschitz constant $\frac12\frac{n+2}{n-2}$, and by the Sobolev inequality,  $\tnorm f-f_0\tnorm_X \geq \S_n^{-1/2}\|f-f_0\|_{2n/(n-2)}$.
Therefore,
\begin{equation}\label{local1f}
\tnorm f-f_0\tnorm_X \geq  \frac{2}{\S_n^{1/2}}\frac{n-2}{n+2}\|\nabla \E(f) - \nabla \E(f_0)\|_{2n/(n+2)}\ .
\end{equation}
Using this in (\ref{local1e}), we obtain (recalling that $Y = L^{2n/(n+2)}$):
\begin{eqnarray}\label{local1ee}
\E^*(g) - \F^*(g) 
&\geq& \frac {n-2}{n+2} \left( \lambda \frac{4}{\S_n} \frac{n-2}{n+2} \|\nabla \E(f) - \nabla \E(f_0)\|_Y^2
+\|g - \nabla \E(f)\|_Y^2\right) - \epsilon  \nonumber\\
&\geq&
 \frac12 \frac{n-2}{n+2} \min\left\{
 \lambda\frac{4}{\S_n} \frac{n-2}{n+2}\ ,\ 1\right\}  \left(
\|g - \nabla \E(f_0)\|_Y \right)^2 - \epsilon\nonumber
\end{eqnarray}
Since $\epsilon>0$ is arbitrary, the proof is complete. 
\end{proof}

\section{Monotonicity and stability}

In this section we discuss how the results of the previous section may be combined with functional flows to obtain quantitative stability results. 

The HLS Sobolev inequalities for $\alpha =1$  are intimately connected with the fast diffusion equation. The content of this section
draws on results from \cite{CCL} and especially \cite{Dem1}.

For a smooth function $f(\eta)$ on $\R^n$, $\Delta f(\eta) = \sum_{j=1}^n \frac{\partial^2}{\partial \eta^2}f(\eta)$.
A smooth function $u(\eta,t)$ on $\R^n\times (0,\infty)$  satisfies the {\em fast diffusion equation} in case
for some $m\in (0,1)$, 
\begin{equation}\label{fd0}
\frac{\partial u}{\partial t}(\eta,t) = \Delta u^m(\eta,t)
\end{equation}
 (For $m=1$, this  is the
usual  diffusion equation).

Simple calculations show 
that $u(\eta,t)$ solves (\ref{fd0}) if and only if
$
v(\eta,t) := e^{td}u(e^t \eta,e^{\beta t})
$
with $\beta=2- d(1-m)$ satisfies the
equation
\begin{equation}\label{fd}
\frac{\partial v}{\partial t}(\eta,t) = \beta \Delta v^m(
\eta,t) +
\nabla \cdot[\eta v(\eta,t)]\ .
\end{equation}
This equation has a steady state:
\begin{equation}\label{fd2A}
v_{\infty,M}(\eta) := \left(D(M) + \frac{1-m}{2\beta m}
|\eta|^2\right)^{-1/(1-m)}\ .
\end{equation}
where $D(M)$ is a computable constant. The rescaling induced by the additional drift term cancels the spreading effect of the diffusion. 
For $m=1$, this is the Fokker-Planck equation, and (\ref{fd}) is a
non-linear relative of the Fokker-Planck equation.

The Cauchy problem for
the FDE (\ref{fd0}) has been studied extensively; see
V\'azquez \cite{V2}.    
Herrero and Pierre \cite{HP} proved that the range of mass
conservation for the fast diffusion equation is $1-2/d < m < 1$, which  is exactly the range of $m<1$ in which
integrable self-similar solutions exist. Within this range, the
flow associated to the fast diffusion equation is in many ways
{\em even better} than the flow associated to the heat equation;
see V\'azquez \cite{V1}. The solutions of
(\ref{fd0}) with positive integrable initial data are $C^\infty$
and strictly positive everywhere instantaneously, just  as for the
heat flow.

Moreover, for non-negative initial data $f$ of mass $M$ satisfying
\begin{equation}\label{assump_inidat}
\sup_{|x|>R} f(x) |x|^{2/(1-m)} <\infty
\end{equation}
for some $R>0$, which means that $f$ is
decaying at infinity at least as fast as the
Barenblatt profile  $v_{\infty,M}$, the solution $v(x,t)$ of  (\ref{fd0})
with initial data $f$ satisfies the following remarkable bounds:
For any $t_*>0$, there exists a constant $C=C(t^*)>0$ such that
\begin{equation}\label{global bound}
\frac1C \leq \frac{v(x,t)}{v_{\infty,M}} \leq C \, ,
\end{equation}
for all $t\geq t_*$ and $x\in\R^d$.  This shows that 
fast diffusion is ``really fast'' in that it  spreads mass out
to infinity to produce the ``right tails'' instantly.

Finally, it is well-known   \cite{BBDGV} that
\begin{equation}\label{liml1}
\lim_{t\to\infty} \|v(t)-v_{\infty,M}\|_{L^1(\R^d)} =0 .
\end{equation}

Now consider the case $m = 1-1/n$, $n\geq 3$. Then $v_{\infty,M}(\eta) = 
\left(D(M) + \frac{1-m}{2\beta m}
|\eta|^2\right)^{-n}$. Thus, $v_{\infty,M}(\eta)^{(n-2)/2n}$ is a Sobolev optimizer (for $\alpha =1$).
Demange \cite{Dem1} shows in that for non-negative $f\in L^{2n/(n-2)}(\R^n)$, $n \geq 3$, if one defines $f_t$ for $t>0$ by taking the $(n-2)/2n$ power of the solution at time $t$, of the Cauchy problem for
(\ref{fd} with initial data $v(\eta) = f^{2n/(n-2)}(\eta)$ then $\S_{n,1}\|\nabla f_t\|_2^2 - \|f_t \|_{2n/(n-2)}^2$ is monotone decreasing in $t$. By what we have explained above, as $t\to\infty$, 
$f_t$ converges to $v_{\infty,M}(\eta)^{(n-2)/2n}$ which is a non-zero Sobolev optimizer. Hence if 
$\S_{n,1}\|\nabla f\|_2^2 - \|f \|_{2n/(n-2)}^2$ is sufficiently small, the local Bianchi-Egnell bound becomes applicable at some large $t$. But then $\S_{n,1}\|\nabla f_t\|_2^2 - \|f_t \|_{2n/(n-2)}^2$
cannot be too small at this point or the local bound would be violated. By Demange's monotonicity result, then  $\S_{n,1}\|\nabla f\|_2^2 - \|f \|_{2n/(n-2)}^2$ cannot have been too small either. In this way, one obtains a {\em global} quantitative version of the Bianchi-Egnell Theorem with a computable lower bound.   By duality, this transfers to the HLS inequality as well.

Monotonicity under fast diffusion with $m= n/(n+2)$  flow for the $\alpha =1$ HLS functional has be proved in 
\cite{CCL}. However, this was done for (\ref{fd0}) and not (\ref{fd}), and there is no limiting optimizer in this case -- except in for the Logarithmic HLS inequality, which has better scaling properties, so that for it there is monotonicity along (\ref{fd}) as well.   This provides another approach to quantitative stability for the Logarithmic HLS inequality. Quantitative stability for the  Logarithmic HLS inequality
under additional assumptions on $f$
has been proved in \cite{CF}, the additional assumptions are natural for the application made there to the Keller-Segel equation, but the application would be strengthened by relaxing them. This will be taken up elsewhere.

\section{Appendix}

We now give a detailed and self-contained proof of Theorem~\ref{sconlp}. The cases $1 < p \leq2$ and the case $p>2$ are treated quite differently, and of course the results are quite different. The inequalities in 
Theorem~\ref{sconlp} are closely related to the uniform convexity properties of the $L^p$ spaces.  A proof of the optimal $2$-uniform convexity inequality for $L^p$, $1 < p \leq$, leads directly to the the $(p-1)$-convexity of the squared $L^p$ norm for such $p$.   For $p>2$, the 
 ``easy'' Clarkson inequality gives a sharp uniform convexity bound, but more work is required to  obtain the
 local $p$th power convexity of  the squared $L^p$ in this range.

\begin{proof}[Proof of Theorem~\ref{sconlp} ]
First  let $1 < p \leq 2$. Let $X = L^p(\Omega,\mathcal{B},\mu)$ as in Example~\ref{lpex}.
What follows is adapted form \cite{BCL}, with more detail provided.
Let $f$ and $g$ be simple functions of the form
$$f(x) = \sum_{j=1}^n z_j1_{A_j}(x)\qquad{\rm and}\qquad  g(x) = \sum_{j=1}^n w_j1_{A_j}(x)\ ,$$
where $\mu(A_j) < \infty$ for each $j$. and for each $j$, $z_jw_j^*$ is not real. 
This latter condition guarantees that $z_j + tw_j \neq 0$ for any real $t$, and thus
for all $x\in\cup_{j=1}^nA_j$, and all $t\in \R$, $f(x)+tg(x) \neq 0$.
Define 
$$Y(t) = \|f+tg\|_p^p \qquad {\rm and}\qquad q = \frac{p}{2}\ ,$$
so that
${\displaystyle \|f+tg\|_p^2 = Y^{1/q}(t)}$.
Differentiating twice,
\begin{equation*}
\frac{{\rm d}^2}{{\rm d}t^2} \|f+tg\|_p^2  \frac{1}{q}\left(\frac{1}{q} -1\right)Y^{1/q-2}(Y')^2 + \frac{1}{q}Y^{1/q-1}Y''\nonumber\\  \frac{1}{q}Y^{1/q-1}Y''\nonumber
\end{equation*}
(Note that already this estimate would fail for $p> 2$.)

Next, a simple calculation, and the fact that $|f+tg|^{2q-2-4}(\Re((f+tg)^* g))2 \leq  
|f+tg|^{2q-2-2}|g|^2$ yields
$$Y''(t) \geq p(p-1) \int |f+tg|^{2q-2}|g|^2 {\rm d}\mu\ .$$
The {\em reverse H\"older inequality} says that for $0 < r < 1$ and $s= r/(r-1)$, whenever $a_j \geq 0$ for $j=1,\dots, n$,
and $b_j > 0$ for $j=1,\dots, n$,
$$
\sum_{j=1}^n a_j b_j \geq \left(\sum_{j=1}^n a_j^r\right)^{1/r}\left(\sum_{j=1}^n b_j^s\right)^{1/s}\ .$$
Applying this to our integral of simple functions, the result is that for all $t$, 
\begin{equation}\label{2ukey}
\frac{{\rm d}^2}{{\rm d}t^2} \|f+tg\|_p^2 \geq 2(p-1)\|g\|_p^2\ .
\end{equation}
Now the restriction the $f$ and $g$ be (special) simple functions is easily removed by density, and 
hence (\ref{2ukey}) is valid for all $f,g \in X$.

A much more direct calculation shows that $f,g\in X$,  
$\frac{{\rm d}}{{\rm d}t} \|f+tg\|_p^2\bigg|_{t=0} = \langle  g, \nabla\E(f)\rangle$.
Therefore, integration yields $\|f+g\|_p^2 \geq \|f\|_p^2 +  \langle g, \nabla\E(f)\rangle + (p-1)\|g\|_p^2$. 
Defining $f_1 = f$ and $f_2 = f+g$, this is equivalent to
\begin{equation}\label{lp2c}
\E(f_2) \geq  \E(f_1)  + \langle f_2-f_1,\nabla\E(f_1)\rangle +  (p-1)\|f_2-f_2\|_p^2\ .
\end{equation}
This completes the proof that $\E$ is $(p-1)$-convex on $X = L^p(\Omega,\mathcal{B},\mu)$ for $1< p \leq 2$. 

For $2 < p < \infty$, we cannot proceed via a lower bound on the second derivative of  $\|f+tg\|_p^2$. 
We shall instead use a duality argument. 
For  $p\geq 2$, let $f\in X$ and $g\in Y$ using the notation of   
Example~\ref{lpex}.  
We showed there that 
\begin{equation}\label{lpp1}
\E(f) + \E^*(g) \geq \langle f,g\rangle - (\|f\|_p - \|g\|_{p'})^2 + 
\left(2\|f\|_p\|g\|_p - 2\Re \left(\int_\Omega f^* g{\rm d}\mu \right)\right)\ .
\end{equation}
Now fix $f_1,f_2\in X$, and let $g= \nabla \E(f_1) = \|f_1\|_p^{2-p}|f_1|^{p-1}{\rm sgn}(f_1)$. 
Then taking $f=f_2$ and this choice of $g$ in (\ref{lpp1}), we obtain
\begin{multline}\label{lpp2}
\E(f_2) + \E^*( \nabla \E(f_1)) \geq \langle f_2, \nabla \E(f_1)\rangle\\ + (\|f_2\|_p- \| \nabla \E(f_1)\|_{p'})^2 + 
\left(2\|f_2\|_p\| \nabla \E(f_1)\|_{p'} - 2\Re \left(\int_\Omega f_2^*  \nabla \E(f_1){\rm d}\mu \right)\right)\ .
\end{multline}
By the cases of equality in Young's inequality,
$$\E(f_1) + \E^*( \nabla \E(f_1)) = \langle f_1, \nabla \E(f_1)\rangle\ .$$
Combining this with (\ref{lpp1}), and the norm preservation  property (\ref{isom}) of $\nabla \E$, we obtain
\begin{multline}\label{lpp3}
\E(f_2)  \geq \E(f_1) + \langle f_2-f_1, \nabla \E(f_1)\rangle \\ + (\|f_2\|_p - \| f_1\|_{p})^2 + 
\left(2\|f_2\|_p\| f_1\|_{p} - 2\Re \left(\int_\Omega f_2^*  \nabla \E(f_1){\rm d}\mu \right)\right)\ .
\end{multline}
Define the unit vectors $u$ and $v$ by  $u= \|f_2\|_p^{-1}f$ and 
\begin{equation}\label{vfor}
v = \nabla \E(\|f_1\|^{-1}f_1) = \|f_1\|_p^{1-p}|f_1|^{p-1}{\rm sgn}(f_1)\ .
\end{equation}
  Then
\begin{equation}\label{lpp4}
\left(2\|f_2\|_p\| f_1\|_{p} - 2\Re \left(\int_\Omega f_2^*  \nabla \E(f_1){\rm d}\mu \right)\right) = 
2\|f_2\|_p\| f_1\|_{p}  \left(1  - \Re \left(\int_\Omega u  v^* {\rm d}\mu \right)\right)\ .
\end{equation}
It is proved in \cite{CFL}, as a simple consequence of the ``easy Clarkson inequality'' that for all unit vectors
$u\in X$ and $v\in Y$,
\begin{equation}\label{horem2}
1   -   \Re \left(\int_\Omega u  v^* {\rm d}\mu \right)   \geq  \frac{1}{p2^{p-1}} \|u - \nabla \E^*(v)\|_p^p\ .
\end{equation}

Next, using (\ref{vfor}),
$$\nabla \E^*(v)  =  |v|^{1/(p-1)}{\rm sgn}(v) =  \|f_1\|_p^{-1} f_1\ .$$
Thus,
\begin{equation}\label{horem3}
2\|f_2\|_p\| f_1\|_{p}  \left(1  - \Re \left(\int_\Omega u  v^* {\rm d}\mu \right)\right) \geq 
\frac{\|f_2\|_p^{1-p}\| f_1\|_{p}^{1-p}}{p2^{p-1}} \|\|f_1\|_p f_2  - \|f_2\|_p f_1\|_p^p\ .
\end{equation}
Thus, (\ref{lpp3}) becomes
\begin{multline}\label{lpp33}
\E(f_2)  \geq \E(f_1) + \langle f_2-f_1, \nabla \E(f_1)\rangle \\ + (\|f_2\|_p - \| f_1\|_{p})^2 + 
\frac{\|f_2\|_p^{1-p}\| f_1\|_{p}^{1-p}}{p2^{p-1}} \|\|f_1\|_p f_2  - \|f_2\|_p f_1\|_p^p\ .
\end{multline}
Define
$$ m := \frac{\|f_1\|_p + \|f_2\|_p}{2}\ ,\quad a := \frac{  \|f_1\|_p - \|f_2\|_p }{\|f_1\|_p + \|f_2\|_p} \quad{\rm and}\quad
\epsilon =  \frac{ \|f_1-f_2\|_p}{\|f_1\|_p + \|f_2\|_p}\ .$$ 
Then 
$$\|f_1\|_p\| f_2\|_p  = m^2(1-a^2)$$
and
$$ \|\|f_1\|_p f_2  - \|f_2\|_p f_1\|_p  = m\| (f_2-f_1) + a(f_1+f_2) \|_p \geq 2m^2(\epsilon -a)\geq 0\ .$$
we can bound the right hand side of (\ref{lpp33}) from below by
as
$$4m^2a^2 +  \frac{2m^{2}(1-a^2)^{1-p}}{p}(\epsilon -a)^p $$
If $a> \tfrac12\epsilon^{p/2}$, this is bounded below by  $m^2\epsilon^p$. If 
$a \leq \tfrac12\epsilon^{p/2}$, this is bounded below by 
$\frac{1}{p}\left(\frac23\right)^{p-1} m^2\epsilon^p$. This proves (\ref{smokra2}).
\end{proof}

 \end{document}